\documentclass[11pt,a4paper]{article}

\usepackage{epsf,epsfig,amsfonts,amsgen,amsmath,amstext,amsbsy,amsopn,amsthm}
\usepackage{amsmath,times,mathptmx}
\usepackage{amsfonts,amsthm,amssymb}
\usepackage{amsfonts}
\usepackage{graphics}
\usepackage{latexsym,bm}
\usepackage{amsfonts,amsthm,amssymb,bbding}
\usepackage{indentfirst}
\usepackage{graphicx}
\usepackage{color}
\usepackage[colorlinks=true,anchorcolor=blue,filecolor=blue,linkcolor=blue,urlcolor=blue,citecolor=blue]{hyperref}
\usepackage{float}
\usepackage{tikz,enumerate}
\evensidemargin=\oddsidemargin\topmargin=-1.5cm

\newtheorem{thm}{Theorem}[section]
\newtheorem{prob}{Problem}[section]

\newtheorem{lem}{Lemma}[section]

\newtheorem{pro}{Proposition}[section]

\newtheorem{remark}{Remark}[section]

\newtheorem{claim}{Claim}[section]
\newtheorem{definition}{Definition}[section]

\addtocounter{section}{0}

\begin{document}
\title{Extremal eigenvalues with respect to graph minors\footnote{Supported by
the National Natural Science Foundation of China (Nos. 12571369, 12271162).}}
\author{{\bf Mingqing Zhai$^a$}, {\bf Longfei Fang$^{b,c}$},
{\bf Huiqiu Lin$^{b}$}\thanks{Corresponding author: huiqiulin@126.com
(H. Lin)}
\\
\small $^{a}$ School of Mathematics and Statistics, Nanjing University of Science and Technology, \\
\small   Nanjing, Jiangsu 210094, China\\
\small $^{b}$ School of Mathematics, East China University of Science and Technology, \\
\small  Shanghai 200237, China \\
\small $^{c}$ School of Mathematics and Finance, Chuzhou University, Chuzhou, Anhui 239012, China
}

\date{}
\maketitle
\begin{abstract}
Minors play a crucial role in various branches of graph theory,
including structural graph theory, extremal graph theory, and topological graph theory, and have garnered significant interest in these areas.
This paper explores the maximal spectral radius, denoted $spex(n,\mathbb{H}_{minor})$, 
of $n$-vertex graphs that exclude any graph from a fixed family $\mathbb{H}$ as a minor. 

We derive the asymptotic value for $spex(n,\mathbb{H}_{minor})$ and establish a unified stable structure for extremal graphs
by introducing a novel application of the absorbing method to eigenvalue analysis, along with models and partitions
with respect to a general $H$ minor.
In particular, we prove three central theorems, the most fundamental of which asserts that 
every graph with spectral radius
$\rho\geq spex(n,\{H\}_{minor})$
contains either an $H$ minor or a spanning book
$B_{\gamma_H,n-\gamma_H}$,
where $\gamma_H\!=\!|H|\!-\!\alpha_H\!-\!1$ and $\alpha_H$ is the independence number of $H$.

These three theorems, combined with detailed combinatorial analysis, enable us to
determine $spex(n,\{H\}_{minor})$ for every complete $r$-partite graph $H$. 
This extends the result of Tait for $spex(n,\{K_r\}_{minor})$ and 
provides a stronger solution to his conjecture for $spex(n,\!\{K_{s,t}\}_{minor})$ [J. Combin. Theory Ser. A 166, 2019].
Additionally, these theorems imply or strengthen other existing eigenvalue-extremal results on minors,
such as planar graphs by Tait and Tobin, $K_r\!-\!E(H)$ minors by Chen, Liu and Zhang,
and friendship graph minors by He, Li and Feng.

\medskip

\noindent {{\bf Keywords:}
graph minor; spectral radius; stability method; absorbing method}

\noindent {\bf AMS\,\,\,Classification:} 05C50; 05C35.
\end{abstract}

\section{Introduction}
Over the past few decades,
research on the spectra of graphs---particularly the eigenvalues of 
their adjacency matrices---has been intensively studied.
Notable contributions include the work of  Alon \cite{Alon}, Bollob\'{a}s, Lee and Letzter \cite{BLL}, Bollob\'{a}s and Nikiforov \cite{BN},
Hoory, Linial and Widgerson \cite{Hoory}, Huang \cite{Huang},
Jiang \cite{Jiang}, Jiang, Tidor, Yao, Zhang and Zhao \cite{Jiang1}, Lubetzky, Sudakov and Vu \cite{LSV}, Tait \cite{Tait}, Tait and Tobin \cite{Tait1}, Wilf \cite{WILF}.

Fix a simple graph $H$. 
We call $H$ a minor of a graph $G$ if $H$ can be obtained from $G$ through a sequence of vertex deletions, edge deletions, and edge contractions.
Given a graph family $\mathbb{H}$, 
a graph is said to be $\mathbb{H}$-minor-free if it does not contain any member of $\mathbb{H}$ as a minor.
Minors play an important role in graph theory.
For a detailed background, refer to the survey by Robertson and Seymour \cite{RS}.
Wagner's theorem states that a graph is planar if and only if it is $\{K_5,K_{3,3}\}$-minor-free (see \cite{WAGNER}).
Planar graphs have been extensively studied, 
with profound findings in areas such as curvatures \cite{Hua1,Gh,Hi,Hua2,Hua3}, automorphism groups \cite{Ba,Con,Ma,Se}, partitions \cite{Cha,Ding,Ked},
and eigenvalues \cite{BOOT,CAO,CP,DM,EZ,Hong1,Hong2,Lin,Tait1}.

As a generalization of planar graphs, 
investigating problems on $K_r$-minor-free or $K_{s,t}$-minor-free graphs has garnered significant interest.
Mader \cite{Mader1} proved that $B_{r-2,n-r+2}$ has the maximum number of edges among
all $K_r$-minor-free graphs on $n$ vertices when $r\leq7$;
however, this result is not optimal for $r>7$.
In 1943, Hadwiger \cite{Hadwiger} conjectured that every $K_r$-minor-free graph is $(r-1)$-colorable for all integers $r\geq 1$.
During the 1980s, Kostochka \cite{Kos,Kos1} and Thomason \cite{Thomason1} independently proved that the maximum number of edges in a $K_r$-minor-free graph $G$
is $\Theta(r\sqrt{log r}\cdot n)$ for sufficiently large $r$,
which implies that $G$ is $O(r\sqrt{\log r})$-colorable.
In 2001, Thomason \cite{Thomason2} determined the asymptotic value of this edge-extremal function.
Recently, Alon, Krivelevich and Sudakov \cite{Alon1} provided a concise and self-contained proof of the Kostochka-Thomason bound.
Norin, Postle and Song \cite{NPS} showed that 
every $K_r$-minor-free graph is $O(r(\log r)^\beta)$-colorable for $\beta>1/4$.
Regarding $K_{s,t}$-minor-free graphs,
Ding, Johnson, and Seymour \cite{DJS} demonstrated that
every connected $K_{1,t}$-minor-free graph $G$ satisfies
$e(G)\leq{t\choose 2}+n-t$.
Chudnovsky, Reed and Seymour \cite{CRS} proved that
$e(G)\leq\frac12(t+1)(n-1)$ for every $K_{2,t}$-minor-free graph $G$,
thereby confirming a conjecture by Myers \cite{MY}.

Let $spex(n,\mathbb{H}_{minor})$ denote the maximal spectral radius of an
$n$-vertex $\mathbb{H}$-minor-free graph.
In the early 1990s, Boots and Royle \cite{BOOT}, and independently Cao and Vince \cite{CAO},
posed a conjecture regarding $spex(n,\{K_5,K_{3,3}\}_{minor})$.
In 2019, Tait and Tobin \cite{Tait1} resolved this long-standing conjecture for sufficiently large $n$.
Tait \cite{Tait} later determined $spex(n,\{K_r\}_{minor})$
and established a sharp upper bound for $spex(n,\{K_{s,t}\}_{minor})$.
Subsequently, Zhai and Lin \cite{ZL} determined the exact value of $spex(n,\{K_{s,t}\}_{minor})$.
In this paper, we address the following fundamental problem.

\begin{prob}\label{prob1.1}
For a graph $H$ (or a graph family $\mathbb{H}$),
what is the maximal spectral radius of all
$H$-minor-free (or $\mathbb{H}$-minor-free) graphs of order $n$?
\end{prob}

A \emph{book}, denoted by $B_{\gamma,n-\gamma}$, is obtained
by joining each vertex of a $\gamma$-clique to every vertex of an independent set of size $n-\gamma$.
For a given graph family $\mathbb{H}$, we define 
\begin{eqnarray}\label{align.01}
\gamma_H=|H|-\alpha_H-1~~
\mbox{and}~~\gamma_\mathbb{H}:=\min_{H\in \mathbb{H}}\gamma_H.
\end{eqnarray}
where $\alpha_H$ denotes the independence number of the graph $H$.

{\flushleft\textbf{Notation.}}
Throughout this paper,
we assume that $n$ is sufficiently large and that
$\mathbb{H}$ contains no member isomorphic to a star.
Consequently, $\gamma_\mathbb{H}\geq1$.

Problem \ref{prob1.1} has garnered significant attention and has been explored by numerous researchers for specific graphs $H$ 
(see, for example, \cite{Hong3,NIKI,Tait,Tait1,WANG,ZL}).
However, a unified framework for addressing Problem \ref{prob1.1} remains elusive.
Let $SPEX(n,\mathbb{H}_{minor})$ denote the family of extremal graphs with respect to $spex(n,\mathbb{H}_{minor})$. 
Our primary result for Problem \ref{prob1.1} is as follows.

\begin{thm}\label{thm1.1}
Every graph in $SPEX(n,\mathbb{H}_{minor})$ contains a spanning book $B_{\gamma_\mathbb{H},n-\gamma_\mathbb{H}}$.
Moreover, 
$$\sqrt{\gamma_\mathbb{H} n}\!+\!\frac{\gamma_\mathbb{H}-1}{2}\!+\!O\big(\frac{1}{\sqrt{n}}\big)
\!\leq\!spex(n,\mathbb{H}_{minor})
\!\leq\!\sqrt{\gamma_\mathbb{H} n}\!+\!\frac{\alpha_\mathbb{H}+\gamma_\mathbb{H}-1}{2}\!+\!O\big(\frac{1}{\sqrt{n}}\big).$$
\end{thm}

We now introduce some new notations and terminologies.
Choose an arbitrary graph $H\in\mathbb{H}.$
We define $\Gamma_s(H)$ to be the set of $s$-vertex induced subgraphs of $H$.
A member $H[S]$ in $\Gamma_s(H)$ is said to be \emph{irreducible},
if $\Gamma_s(H)$ does not contain any member isomorphic to a proper subgraph of $H[S]$.
Now, let $\Gamma_s^*(H)$ be the subset of  $\Gamma_s(H)$ in which every member is irreducible, 
and set
$$\Gamma(\mathbb{H})=\bigcup_{H\in\mathbb{H}}\Gamma_{|H|-\gamma_\mathbb{H}}^*(H).$$
It is clear that
$\Gamma_{|H|-\gamma_\mathbb{H}}^*(H)=\Gamma_{\alpha_H+1}^*(H)$
for every $H\in\mathbb{H}$
with $\gamma_H=\gamma_\mathbb{H}.$

A graph is said to be \emph{$\mathbb{H}$-minor-saturated},
if it is $\mathbb{H}$-minor-free, but the addition of any edge results in a graph containing some member of $\mathbb{H}$ as a minor.
Let $SAT(n,\mathbb{H}_{minor})$
be the set of $n$-vertex $\mathbb{H}$-minor-saturated graphs.
Choose an arbitrary $G^*\in SPEX(n,\mathbb{H}_{minor})$,
and let $L$ be the set of dominating vertices in $G^*$.
By Theorem \ref{thm1.1}, we have $|L|=\gamma_\mathbb{H}$.
The next objective is to characterize $G^*-L$.
We now state the following two theorems.

\begin{thm}\label{thm1.2}
The induced subgraph $G^*-L\in
SAT(n-\gamma_\mathbb{H},\Gamma(\mathbb{H})_{minor})$.
Specifically, if $\mathbb{H}=\{H\}$, then $G^*-L\in SAT(n-\gamma_H,\Gamma^*_{\alpha_H+1}(H)_{minor})$.
\end{thm}

Let $ex(n,\mathbb{H}_{minor})$ denote the maximal number of edges in an
$n$-vertex $\mathbb{H}$-minor-free graph, and let $EX(n,\mathbb{H}_{minor})$
be the set of extremal graphs with respect to $ex(n,\mathbb{H}_{minor})$.
If the members in $\Gamma(\mathbb{H})$ are connected,
we obtain a stronger result than Theorem \ref{thm1.2}.

\begin{thm}\label{thm1.3}
If $\Gamma(\mathbb{H})$ is a connected family,
then $G^*-L\in EX(n-\gamma_\mathbb{H},\Gamma(\mathbb{H})_{minor})$.
\end{thm}

Theorems \ref{thm1.1}, \ref{thm1.2} and \ref{thm1.3}
are crucial tools for addressing Problem \ref{prob1.1}.
An important application of these theorems is to characterize
$SPEX(n,\{H\}_{minor})$ for a complete $r$-partite graph $H$,
which extends the result of Tait for $SPEX(n,\{K_r\}_{minor})$ and 
provides a stronger solution to his conjecture for $SPEX(n,\{K_{s,t}\}_{minor})$ (see \cite{Tait}).

Before stating the result, we need to introduce some notations.
We denote by $\overline{G}$ the complement of a graph $G$, and by $K_{s_1,\ldots,s_r}$, a complete $r$-partite graph
with $s_1\geq\cdots\geq s_r$.
Let $H_{s_1,s_2}=(\beta-1)K_{1,s_2}\cup K_{1,s_2+\beta_0}$,
where $0\leq \beta_0\leq s_2$ and $\beta(s_2+1)+\beta_0=s_1+1$.
Clearly, $H_{s_1,s_2}$ is a star forest of order $s_1+1$.
Let $S\big(\overline{H_{s_1,s_2}}\big)$ denote the graph obtained
from $\overline{H_{s_1,s_2}}$ by subdividing
an edge $uv$ with the minimum degree sum $d(u)+d(v)$.

Denote by ${Pet}^\star$ the Petersen graph.
Set $\beta=\lfloor\frac{s_1+1}{s_2+1}\rfloor$ and $n-\sum_{2}^{r}s_i+1:=ps_1+q$ \emph{($1\leq q\leq s_1$)}.
We now define a graph $G^\blacktriangle$,
which will be used in our next theorem.
\begin{eqnarray}\label{align.00}
G^\blacktriangle=\left\{
\begin{aligned}
   &(p-1)K_{s_1} \cup S\left(\overline{H_{s_1,s_2}}\right) &&\hbox{if $(q,\beta)=(2,2)$};\\
   &(p-1)K_{s_1} \cup \overline{{Pet}^\star} &&\hbox{if $(q,\beta,s_1)=(2,1,8)$}; \\
   &(p-q)K_{s_1} \cup q\overline{H_{s_1,s_2}} &&\hbox{if $q\leq 2(\beta-1)$ and $(q,\beta)\neq(2,2)$}; \\
   &pK_{s_1} \cup K_q &&\hbox{if $q>2(\beta-1)$ and $(q,\beta,s_1)\neq(2,1,8)$}.
\end{aligned}
\right.
\end{eqnarray}


\begin{thm}\label{thm1.8}
Let $r\geq2$, $s_1\geq\cdots\geq s_r\geq1$
and $\gamma=\sum_{i=2}^rs_i-1\geq1$.

\noindent (i) If $s_1$ is even or $s_2\geq2$,
then $SPEX\big(n,\{K_{s_1,\ldots,s_r}\}_{minor}\big)
=\{K_\gamma\nabla G^\blacktriangle\}$.

\noindent (ii) If $s_1$ is odd and $s_2=1$, then $SPEX\big(n,\{K_{s_1,\ldots,s_r}\}_{minor}\big)
=\{K_\gamma\nabla G^\blacktriangledown\}$,
where every component of $G^\blacktriangledown$ is a cycle for $s_1=3$,
and is either $K_{s_1}$ or $\overline{H_{s_1,1}}$ otherwise.
\end{thm}


Theorems \ref{thm1.1} and \ref{thm1.2}
also imply or strengthen other existing results on minors.
Let $G^*\in SPEX(n,\mathbb{H}_{minor})$.
We now proceed with several direct applications.

If $\mathbb{H}=\{K_{a,3},K_{a+2}\}$, where $a\in\{2,3\}$,
then $\gamma_\mathbb{H}=a-1$.
For $H=K_{a,3}$, we have $\Gamma^*_{|H|-\gamma_\mathbb{H}}(H)
=\Gamma^*_4(H)=\{K_{1,3},K_{2,2}\}$, and for $H=K_{a+2}$, $\Gamma^*_{|H|-\gamma_\mathbb{H}}(H)
=\Gamma^*_3(H)=\{K_3\}$. By Theorems \ref{thm1.1} and \ref{thm1.2},
$G^*$ has $a-1$ dominating vertices and $G^*-L$ is $\{K_{1,3},K_{2,2},K_3\}$-minor-saturated.
Therefore, $G^*-L$ is a path, proving the following theorem.

\begin{thm}\label{thm1.4}\emph{(Tait and Tobin, \cite{Tait1})} 
For $a\in\{2,3\}$, we have $$SPEX(n,\{K_{a,3},K_{a+2}\}_{minor})=\{K_{a-1}\nabla P_{n-a+1}\}.$$
\end{thm}


Recently, Chen, Liu and Zhang \cite{Zhang}
characterized $SPEX(n,\{K_r-E(H)\}_{minor})$,
where $H$ is a subgraph of $K_r$ consisting of some vertex-disjoint paths.
Now, let $H$ be an arbitrary subgraph of $K_r$ with clique number $\omega(H)$, 
and let $B^{k}_{s,t}$ denote the graph obtained from the book
$B_{s,t}$ by adding $k$ isolated edges within its independent set.
By Theorems \ref{thm1.1} and \ref{thm1.2}, 
it is straightforward to verify the following extended result.

\begin{thm}\label{thm1.5}
Every graph $G^*$ in $SPEX(n,\{K_r-E(H)\}_{minor})$
contains a spanning book $B_{\gamma,n-\gamma}$, where $\gamma=r-\omega(H)-1$.
In particular, if $\omega(H)=2$, then
$G^*\cong B^{\lfloor(n-r+3)/2\rfloor}_{r-3,n-r+3}$ for $H\cong \frac{|H|}2K_2$,
and $G^*\cong B_{r-3,n-r+3}$ otherwise.
\end{thm}

A \emph{$(k+1)$-wheel} is defined as $W_{k+1}=K_1\nabla C_k$.
Cioab\u{a}, Desai, Tait \cite{CB1} proved a nice eigenvalue-extremal result for graphs excluding a fixed wheel.
By applying Theorems \ref{thm1.1} and \ref{thm1.2}, 
we immediately obtain a refinement of their result for $W_{k+1}$-minor-free graphs.


\begin{thm}\label{thm1.6}
Let $\gamma=\lceil\frac{k}2\rceil$.
If $k$ is odd, then $SPEX(n,\{W_{k+1}\}_{minor})
=\{B_{\gamma,n-\gamma}\}$.
Otherwise, we have $SPEX(n,\{W_{k+1}\}_{minor})=\{B^1_{\gamma,n-\gamma}\}$.
\end{thm}

A \emph{$t$-flower} $F_{s_1,\ldots,s_t}$
is the graph obtained from $t$ cycles of lengths
$s_1,\ldots,s_t$ respectively by identifying one vertex.
If $s_1=\cdots=s_t=3$, then it is
a friendship graph.
Recently, He, Li and Feng \cite{FENG}
determined $SPEX\big(n,\{F_{s,\ldots,s}\}_{minor}\big)$ for
$s\in\{3,4\}$.
By Theorems \ref{thm1.1} and \ref{thm1.2}, 
we can similarly obtain a stronger result for $F_{s_1,\ldots,s_t}$-minor-free graphs.


\begin{thm}\label{thm1.7}
Let $\gamma=\sum_{i=1}^t\lceil\frac{s_i}2\rceil-t$. 
If there exists an odd $s_i$, then $SPEX(n,\{F_{s_1,\ldots,s_t}\}_{minor})=\{B_{\gamma,n-\gamma}\}$.
Otherwise, we have $SPEX(n,\{F_{s_1,\ldots,s_t}\}_{minor})
=\{B^1_{\gamma,n-\gamma}\}$.
\end{thm}

The rest of the paper is organized as follows.
Section \ref{section2} presents some preliminary results, including a stability theorem, 
which serves as a key tool for this paper.
For the sake of readability,
we postpone the proof of this theorem until the final section.
In Section \ref{section3},
we base our approach on the stability theorem and apply an absorbing method to prove Theorem \ref{thm1.1}.
Building on Theorem \ref{thm1.1},
we can further prove Theorems \ref{thm1.2} and \ref{thm1.3},
which provide a more refined description of graphs in $SPEX(n,\mathbb{H}_{minor})$.
In Section \ref{section4},
we will use Theorems \ref{thm1.1}, \ref{thm1.2} and \ref{thm1.3}
to characterize $SPEX(n,\mathbb{H}_{minor})$ for
complete multipartite minors, thereby proving Theorems \ref{thm1.8}.

\section{Preliminary results}\label{section2}

Observe that isolated vertices in a graph $H$ 
do not affect the determination of whether a graph $G$ of sufficiently large order is $H$-minor-free.
Throughout the paper,
we assume that every member $H$ of the family $\mathbb{H}$ is a finite graph with minimum degree $\delta(H)\geq 1$.
We use $|G|$ to denote the order and $e(G)$ to denote the size of $G$.
Let $G[S]$ be the subgraph of $G$ induced by a vertex subset $S$.

In 1967, Mader \cite{Mader}
proved a landmark result: $e(G)\leq C n$ for
every $n$-vertex $H$-minor-free graph $G$, where $C$ is a constant.
Thomason \cite{Thomason} later gave the following result.

\begin{lem}\label{lem2.1}
Every non-empty graph $G$ with $e(G)\geq2^{s+1}t|G|$
has a proper $K_{s,t}$ minor.
\end{lem}

Lemma \ref{lem2.1} can be extended to general graph minors.

\begin{lem}\label{lem2.2}
Let $G$ be an $H$-minor-free graph of order $n$. Then
$e(G)<2^{|H|+1}e(H) n$.
\end{lem}

\begin{proof}
Let $H'$ be the graph obtained from $H$ by subdividing each edge once.
Clearly, $H'$ is a bipartite spanning subgraph of $K_{|H|,e(H)}$.
Thus, $K_{|H|,e(H)}$ contains $H$ as a minor.

Now, suppose to the contrary that $e(G)\geq2^{|H|+1}e(H)n.$
By Lemma \ref{lem2.1}, $G$ contains a $K_{|H|,e(H)}$ minor,
and thus an $H$ minor, a contradiction. Therefore, the result follows.
\end{proof}


\begin{lem}\label{lem3.1} \emph{(\cite{CRS,DJS})}
Let $t\geq 3$ and $n\geq t+2$.
If $G$ is an $n$-vertex connected graph with no $K_{1,t}$ minor,
then $e(G)\leq \binom{t}{2}+n-t$,
and for all $n$, this is the best possible.
\end{lem}

Let $S^\ell(K_t)$ denote the graph obtained from $K_t$ by subdividing one edge $\ell$ times.
As noted by Ding, Johnson and Seymour \cite{DJS},
the upper bound in Lemma \ref{lem3.1} is sharp,
and $S^{n-t}(K_t)$ serves as an extremal graph.
A natural question arises: how can we characterize all such extremal graphs?
When $t=3$, $S^{n-t}(K_t)$ is an $n$-cycle,
which is clearly the unique extremal graph.
The next case, $t=4$, will be instrumental for our main theorem.

\begin{lem}\label{lem3.2}
Let $n\geq 5$ and let $G$ be a connected graph of order 
$n$ that is $K_{1,4}$-minor-free, with the maximal number of edges. 
Then $G\cong S^{n-4}(K_4)$.
\end{lem}

\begin{proof}
Since $S^{n-4}(K_4)$ is $K_{1,4}$-minor-free, we have $e(G)\geq e(S^{n-4}(K_4))=n+2$.
On the other hand, for $n\geq 6$, we obtain $e(G)\leq n+2$ by Lemma \ref{lem3.1},
and for $n=5$, we can see that $e(G)\leq \lfloor\frac{3\times 5}{2}\rfloor=7$ since $\Delta(G)\leq3$.
Thus, $e(G)=n+2$.

Let $U_i$ be the set of vertices of degree $i$ in $G$.
Then, we have
$$2n-|U_1|+|U_3|=|U_1|+2|U_2|+3|U_3|=2e(G)=2n+4,$$
which yields $|U_3|-|U_1|=4$.

We now prove that $G\cong S^{n-4}(K_4)$ by induction on $n$.
For the base case $n=5$,
combining $|U_1|+|U_2|+|U_3|=5$ and $|U_3|-|U_1|=4$ gives $2|U_1|+|U_2|=1$.
Thus, $|U_1|=0, |U_2|=1$ and $|U_3|=4$.
Assume that $U_2=\{u\}$ and $U_3=\{v_1,v_2,v_3,v_4\}$, where $N_G(u)=\{v_1,v_2\}$.
Then, $N_G(v_3)=\{v_1,v_2,v_4\}$ and $N_G(v_4)=\{v_1,v_2,v_3\}$.
It follows that $G\cong S^1(K_4)$.

For the inductive step, assume $n\geq 6$.
We first prove that $U_1=\varnothing$.
Suppose, to the contrary that, $u\in U_1$.
Then, $G\!-\!\{u\}$ is a $K_{1,4}$-minor-free connected graph with $n+1$ edges.
By the induction hypothesis, $G\!-\!\{u\}\cong S^{n-5}(K_4)$.
However, regardless of whether $u$ is adjacent to a vertex of degree two or three in $G\!-\!\{u\}$,
$G$ always contains a $K_{1,4}$ minor, a contradiction.
Thus, $U_1=\varnothing$,
which implies that $|U_3|=|U_1|+4=4$ and $|U_2|=n-4\geq 2$.
Let $U_3=\{v_1,v_2,v_3,v_4\}$.
In the following, we distinguish two cases.

\begin{figure}
\centering
\begin{tikzpicture}[x=0.80mm, y=0.80mm, inner xsep=0pt, inner ysep=0pt, outer xsep=0pt, outer ysep=0pt]
\path[line width=0mm] (71.44,60.45) rectangle +(120.20,39.83);
\definecolor{L}{rgb}{0,0,0}
\path[line width=0.30mm, draw=L] (80.00,90.00) -- (80.00,70.00);
\path[line width=0.30mm, draw=L] (80.00,70.00) -- (120.00,70.00);
\path[line width=0.30mm, draw=L] (120.00,90.00) -- (120.00,70.00);
\path[line width=0.30mm, draw=L] (80.00,90.00) -- (93.33,90.00);
\path[line width=0.30mm, draw=L] (120.00,90.00) -- (106.67,90.00);
\path[line width=0.15mm, draw=L] ;
\path[line width=0.30mm, draw=L, dash pattern=on 0.30mm off 0.50mm] (93.33,90.00) -- (106.67,90.00);
\definecolor{F}{rgb}{0,0,0}
\path[line width=0.30mm, draw=L, fill=F] (100.00,70.00) circle (1.00mm);
\path[line width=0.30mm, draw=L, fill=F] (80.00,70.00) circle (1.00mm);
\path[line width=0.30mm, draw=L, fill=F] (120.00,70.00) circle (1.00mm);
\path[line width=0.30mm, draw=L, fill=F] (80.00,90.00) circle (1.00mm);
\path[line width=0.30mm, draw=L, fill=F] (120.00,90.00) circle (1.00mm);
\path[line width=0.30mm, draw=L, fill=F] (93.33,90.00) circle (1.00mm);
\path[line width=0.30mm, draw=L, fill=F] (106.67,90.00) circle (1.00mm);
\path[line width=0.30mm, draw=L] (80.00,70.00) -- (120.00,90.00);
\path[line width=0.30mm, draw=L] (80.00,90.00) -- (120.00,70.00);
\path[line width=0.30mm, draw=L] (140.00,90.00) -- (140.00,70.00);
\path[line width=0.30mm, draw=L] (140.00,70.00) -- (180.00,70.00);
\path[line width=0.30mm, draw=L] (180.00,70.00) -- (180.00,90.00);
\path[line width=0.30mm, draw=L] (140.00,90.00) -- (153.33,90.00);
\path[line width=0.30mm, draw=L] (166.67,90.00) -- (180.00,90.00);
\path[line width=0.30mm, draw=L, dash pattern=on 0.30mm off 0.50mm] (153.33,90.00) -- (166.67,90.00);
\path[line width=0.30mm, draw=L, fill=F] (140.00,70.00) circle (1.00mm);
\path[line width=0.30mm, draw=L, fill=F] (180.00,70.00) circle (1.00mm);
\path[line width=0.30mm, draw=L, fill=F] (140.00,80.00) circle (1.00mm);
\path[line width=0.30mm, draw=L, fill=F] (140.00,90.00) circle (1.00mm);
\path[line width=0.30mm, draw=L, fill=F] (153.33,90.00) circle (1.00mm);
\path[line width=0.30mm, draw=L, fill=F] (166.67,90.00) circle (1.00mm);
\path[line width=0.30mm, draw=L, fill=F] (180.00,90.00) circle (1.00mm);
\path[line width=0.30mm, draw=L] (140.00,70.00) -- (180.00,90.00);
\path[line width=0.30mm, draw=L] (140.00,90.00) -- (180.00,70.00);
\draw(78,92.4) node[anchor=base west]{\fontsize{14.23}{17.07}\selectfont $v_1$};
\draw(118.66,92.4) node[anchor=base west]{\fontsize{14.23}{17.07}\selectfont $v_2$};
\draw(78,65.5) node[anchor=base west]{\fontsize{14.23}{17.07}\selectfont $v_3$};
\draw(118,65.5) node[anchor=base west]{\fontsize{14.23}{17.07}\selectfont $v_4$};
\draw(98.6,65.5) node[anchor=base west]{\fontsize{14.23}{17.07}\selectfont $u$};
\draw(138,92.4) node[anchor=base west]{\fontsize{14.23}{17.07}\selectfont $v_1$};
\draw(178,92.4) node[anchor=base west]{\fontsize{14.23}{17.07}\selectfont $v_2$};
\draw(138,65.5) node[anchor=base west]{\fontsize{14.23}{17.07}\selectfont $v_3$};
\draw(178,65.5) node[anchor=base west]{\fontsize{14.23}{17.07}\selectfont $v_4$};
\draw(142,79) node[anchor=base west]{\fontsize{14.23}{17.07}\selectfont $u$};
\draw(90.6,92.4) node[anchor=base west]{\fontsize{14.23}{17.07}\selectfont $w_1$};
\draw(150.42,92.4) node[anchor=base west]{\fontsize{14.23}{17.07}\selectfont $w_1$};
\draw(103.3,92.4) node[anchor=base west]{\fontsize{14.23}{17.07}\selectfont $w_{n-5}$};
\draw(164.43,92.4) node[anchor=base west]{\fontsize{14.23}{17.07}\selectfont $w_{n-5}$};
\end{tikzpicture}
\caption{The extremal graph $G$ in Case 2.}{\label{figure.1}}
\end{figure}
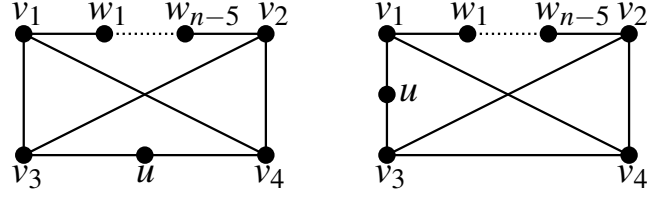

{\textbf{Case 1. Every vertex in $U_2$ is contained in a triangle.}}
\vspace{1mm}

Recall that $|U_2|\geq2$.
Choose $u_1,u_2\in U_2$.
Suppose first that $u_1,u_2$ share two common neighbors, $w_1$ and $w_2$.
Since $u_i$ belongs to
a triangle for $i\in \{1,2\}$,
it follows that $w_1w_2\in E(G)$.
Now, $d_G(w_1)=d_G(w_2)=3$, and hence, $G[\{u_1,u_2,w_1,w_2\}]$ is a component of $G$,
which contradicts the fact that $G$ is a connected graph of order $n\geq6$.

Suppose now that $N_G(u_1)\cap N_G(u_2)=\{w\}$.
We first claim that $u_1u_2\in E(G)$. 
If not, there exists $w_i\in N_G(u_i)\setminus\{w\}$ for $i\in\{1,2\}$.
Clearly, $\{w_1,w_2\}\cap \{u_1,u_2\}=\varnothing$.
Since $u_i$ belongs to a triangle,
we have $ww_i\in E(G)$,
and hence, $d_G(w)\geq4$,
which contradicts $\Delta(G)\leq3$.
Thus, $u_1u_2\in E(G)$. 
Since $n\geq6$ and $G$ is connected, $w$ must have a third neighbor $u_3$ with $d_G(u_3)\geq2$.
Furthermore, $u_3\in U_2$ (otherwise,
$G$ would contain a $K_{1,4}$ minor).
Let $N_G(u_3)=\{w,u_4\}$.
Then, $wu_4\in E(G)$,
which implies $d_G(w)\geq4$, a contradiction.

We now have $N_G(u_1)\cap N_G(u_2)=\varnothing$,
and $u_1u_2\notin E(G)$ since $u_1$ lies in a triangle.
Let $P=u_1w_1\dots w_su_2$ be a shortest $(u_1,u_2)$-path,
and let $N_G(u_1)=\{w_1,u_3\}$ and $N_G(u_2)=\{w_s,u_4\}$.
Then, $u_3\neq u_4$ and $u_3,u_4\notin V(P)$.
Furthermore, $w_1u_3,w_su_4\in E(G)$, as $u_1,u_2\in U_2$.
Thus, $u_1,u_3\in N_G(w_1)$ and $u_2,u_4\in N_G(w_s)$.
By contracting the subpath $w_1\ldots w_s$ as a vertex,
we obtain a copy of $K_{1,4}$, a contradiction.
The proof of Case 1 is complete.

\vspace{1mm}
{\textbf{Case 2. There exists a vertex $u\in U_2$ with two non-adjacent neighbors.}}
\vspace{1mm}

Let $u_1,u_2$ be two non-adjacent neighbors of $u$,
and let $G'$ be the graph obtained from $G$ by contracting the path $u_1uu_2$
as an edge $u_1u_2$.
Then, $e(G')=e(G)-1=|G'|+2$.
Since $G$ is $K_{1,4}$-minor-free, $G'$ is too.
By the induction hypothesis, $G'\cong S^{n-5}(K_4)$.
Let $P$ be the induced path of length $n-4$ in $G'$.
Then, both ends are of degree three in $G'$ and $G$.
We may assume $P=v_1w_1\ldots w_{n-5}v_2$,
where $v_1,v_2\in U_3$.

Suppose that $u_1u_2\in E(G')\setminus E(P)$.
Then, $u_1,u_2$ are of degree three in both $G'$ and $G$.
Now, if $\{u_1,u_2\}=\{v_3,v_4\}$, then $G[\{v_1,v_2,v_3,v_4,w_1,u\}]$ contains
a $K_{1,4}$ minor (see Fig. \!\!\ref{figure.1}), a contradiction.
Thus, $u_1\in\{v_1,v_2\}$ and $u_2\in\{v_3,v_4\}$.
Without loss of generality, assume $\{u_1,u_2\}=\{v_1,v_3\}$.
In this case, $G[\{v_1,v_2,v_3,v_4,w_1,u\}]$ still contains a $K_{1,4}$ minor,
leading to a contradiction.
Therefore, $u_1u_2\in E(P)$, and
thus $G\cong S^{n-4}(K_4)$.
\end{proof}

A member $H$ of $\mathbb{H}$ is said to be \emph{minimal},
if: (i) $\gamma_{H}=\gamma_\mathbb{H}$; and (ii) subject to (i), $|H|$ is also minimal.
It is evident that all minimal graphs have the same independence number. 
Therefore, we can set $\alpha_\mathbb{H}:=\alpha_{H^*}$
for any minimal graph $H^*$.
Furthermore, we define
$C_\mathbb{H}:=\min C_{H^*}$, where $C_{H^*}=2^{|H^*|+1}e(H^*)$ and $H^*$ ranges over all minimal graphs.
Recall that $\gamma_{H^*}=|H^*|-\alpha_{H^*}-1$.
Hence, we have $\gamma_\mathbb{H}+\alpha_\mathbb{H}<C_\mathbb{H}.$

\begin{lem}\label{lem2.3}
$B_{\gamma_H+1,\alpha_H}$ contains an $H$ minor.
\end{lem}

\begin{proof}
Observe that $H$ is a spanning subgraph of $B_{\gamma_H+1,\alpha_H}$.
The statement follows.
\end{proof}

An \emph{elementary operation} on a graph is one of the following:
deleting a vertex, deleting an edge, or contracting an edge.
Clearly, a graph $G$ contains an $H$ minor
if $H$ can be obtained from $G$ by a sequence of elementary operations.

\begin{lem}\label{lem2.4}
$B_{\gamma_{\mathbb{H}},n-\gamma_{\mathbb{H}}}$ is $\mathbb{H}$-minor-free,
and $spex(n,\mathbb{H}_{minor})\geq\sqrt{\gamma_{\mathbb{H}}(n-\gamma_{\mathbb{H}})}$.
\end{lem}

\begin{proof}
Choose an arbitrary member $H\in\mathbb{H}$.
We first claim that $B_{\gamma_H,n-\gamma_H}$ contains no copy of $H$.
Indeed, if $H\subseteq B_{\gamma_H,n-\gamma_H}$,
let $T$ be the $\gamma_H$-clique and $S$ be the independent set of $n-\gamma_H$ vertices 
in $B_{\gamma_{H},n-\gamma_{H}}$.
Then, we have
$$|V(H)\cap S|=|V(H)\setminus T|\geq|H|-\gamma_H=\alpha_{H}+1.$$
Furthermore, since $V(H)\cap S$ is also an independent set in $H$,
we obtain $\alpha_{H}\geq|V(H)\cap S|\geq\alpha_{H}+1$, a contradiction.
Therefore, the claim holds.

We now show that $B_{\gamma_{H},n-\gamma_{H}}$ is $H$-minor-free.
Suppose, to the contrary, that $B_{\gamma_{H},n-\gamma_{H}}$ contains an $H$ minor.
Then, $H$ can be obtained from $B_{\gamma_{H},n-\gamma_{H}}$ by a sequence of elementary operations.
These elementary operations give rise to a graph sequence
$H_0,H_1,\cdots,H_a$, where $H_a=H$.
From the structure of a book,
we know that every elementary operation
on a subgraph of $B_{\gamma_{H},n-\gamma_{H}}$
results in a new subgraph of $B_{\gamma_{H},n-\gamma_{H}}$.
This implies that $H\subseteq B_{\gamma_{H},n-\gamma_{H}}$,
contradicting the claim proved above.
Thus, $B_{\gamma_{H},n-\gamma_{H}}$ is $H$-minor-free.

In view of (\ref{align.01}), we have
$\gamma_{\mathbb{H}}\leq\gamma_H$, and thus
$B_{\gamma_{\mathbb{H}},n-\gamma_{\mathbb{H}}}\subseteq B_{\gamma_{H},n-\gamma_{H}}$.
Since $B_{\gamma_{H},n-\gamma_{H}}$ is $H$-minor-free, $B_{\gamma_{\mathbb{H}},n-\gamma_{\mathbb{H}}}$
must also be $H$-minor-free.
Considering the choice of $H$, we can conclude that
$B_{\gamma_{\mathbb{H}},n-\gamma_{\mathbb{H}}}$
is $\mathbb{H}$-minor-free,
and so is $K_{\gamma_{\mathbb{H}},n-\gamma_{\mathbb{H}}}$.
Hence,
$spex(n,\mathbb{H}_{minor})\geq
\rho(K_{\gamma_{\mathbb{H}},n-\gamma_{\mathbb{H}}})
=\sqrt{\gamma_{\mathbb{H}}(n-\gamma_{\mathbb{H}})}$.
\end{proof}

We conclude this section with the following stability result. 
For the sake of readability, we postpone its proof until the last section.

\begin{thm}\label{thm2.1}
Let $G$ be a graph of sufficiently large order $n$.
Let $X$ be a non-negative eigenvector corresponding to $\rho(G)$
with $x_{u^*}=\max_{u\in V(G)}x_u$.
If $\rho(G)\geq\sqrt{\gamma_{\mathbb{H}}(n-\gamma_{\mathbb{H}})}$,
then either $G$ contains an $H$ minor for some $H\in\mathbb{H}$,
or there exists a set $L$ of exactly $\gamma_\mathbb{H}$ vertices
such that $x_u\ge\big(1-\frac{1}{2(10C_\mathbb{H})^2}\big)x_{u^*}$
and $d_G(u)\ge\big(1-\frac{1}{(10C_\mathbb{H})^2}\big)n$
for every $u\in L$.
\end{thm}

\section{An absorbing method for characterizing $SPEX(n,\mathbb{H}_{minor})$}\label{section3}

In this section, we prove Theorems \ref{thm1.1}, \ref{thm1.2} and \ref{thm1.3}. 
Choose an arbitrary $G^*\in SPEX(n,\mathbb{H}_{minor})$.
Let $X^*=(x_1,\ldots,x_n)^T$ be a non-negative unit eigenvector corresponding to $\rho(G^*)$,
and let $u^*$ be a vertex in $V(G^*)$ such that $x_{u^*}=\max_{u\in V(G^*)}x_u$.
Define $\rho^*:=\rho(G^*)$.
By Lemma \ref{lem2.4}, we have $\rho^*\geq
\sqrt{\gamma_{\mathbb{H}}(n-\gamma_{\mathbb{H}})}$.
Furthermore, by Theorem \ref{thm2.1}, we obtain the following proposition.

\begin{pro}\label{pro3.0}
$G^*$ contains a set $L$ of exactly $\gamma_\mathbb{H}$ vertices
such that $x_u\ge\big(1-\frac{1}{2(10C_\mathbb{H})^2}\big)x_{u^*}$
and $d_{G^*}(u)\ge\big(1-\frac{1}{(10C_\mathbb{H})^2}\big)n$
for every $u\in L$.
\end{pro}

We now partition $V(G^*)\setminus L$ into $L'\cup L''$,
where $L''=\{v:~ L\subseteq N_{G^*}(v)\}$.
The key point for Theorem \ref{thm1.1} is to show that $L'=\varnothing$,
which will be proved using an {\it absorbing method}.
Specifically,
we will find an {\it absorbing set} within $L''$,
and then use it to absorb vertices in $L'$.
To this end,
we require some auxiliary definitions and propositions.

\begin{definition}\label{de3.1}
A path $P=v_1v_2\ldots v_s$ (where $v_1=v_s$ is allowed) is called a
{\bf linear path} in $G$, if $P\subseteq G$ and $d_G(v_i)=2$
for each $i\in\{2,\ldots s-1\}$.
A linear path $P$ is said to be {\bf maximal}, if there exists
no any linear path $P'$ such that $P\subseteq P'$ and $P\neq P'$.
\end{definition}

By Definition \ref{de3.1},
every linear path in
$G$ is either an induced path or an induced cycle.
For a connected graph $G$ with $|G|\geq2$,
Definition \ref{de3.1} further implies the following two propositions.

\begin{pro}\label{pro3.1}
Let $P$ be a maximal linear path in $G$.
If $P$ is a path, then each endpoint $v$ satisfies
$d_G(v)\neq2$;
if $P$ is a cycle, then at most one vertex $v$ on $P$ satisfies $d_G(v)\neq2$.
\end{pro}

\begin{pro}\label{pro3.2}
Every connected graph has an edge-decomposition
of its maximal linear paths.
\end{pro}

For a graph $H$ with $V(H)=\{v_1,\ldots,v_h\}$,
a {\bf model} of $H$ in a graph $G$ is a collection of vertex-disjoint
connected subgraphs $G_{v_1}$, \ldots, $G_{v_h}$ such that
for any $v_iv_j\in E(H)$, there exists an edge with one end in $G_{v_i}$
and the other in $G_{v_j}$.
It is not hard to see that $G$ has an $H$ minor if and only if
there is a model of $H$ in $G$.
Based on this terminology,
we can introduce the following definition.

\begin{definition}\label{de3.2}
Let $G$ be a graph with an $H$ minor, and let $\{G_{v_i}:~ v_i\in V(H)\}$ be
a model of $H$ in $G$.
Then $(V(G_{v_1}),\dots,V(G_{v_h}))$
is called an {\bf $H$-partition} of $G$.
\end{definition}

Note that $\cup_{v_i\in V(H)}V(G_{v_i})\subseteq V(G)$.
Hence, an $H$-partition of $G$ need not form a partition of $V(G)$,
even though its members are pairwise vertex-disjoint.
An $H$-partition is called \emph{minimal},
if $\sum_{i=1}^{|H|}|G_{v_i}|$ is minimized among all $H$-partitions of $G$.

By Proposition \ref{pro3.0},
$|L|=\gamma_{\mathbb{H}}=\min_{H\in\mathbb{H}}(|H|-\alpha_{H}-1)$,
and every vertex in $L$ has at most
$\frac{n}{(10C_\mathbb{H})^2}$ non-neighbors in $L'$. Hence,
\begin{eqnarray}\label{align.09}
|L'|\leq\frac{n}{(10C_\mathbb{H})^2}|L|
=\frac{\gamma_\mathbb{H}n}{(10C_\mathbb{H})^2}.
\end{eqnarray}

Before proceeding, we need some additional notations.
For a graph $G$ with $u\in V(G)$ and
$S\subseteq V(G)$,
we write $N_{S}(u):=N_G(u)\cap S$ and $d_{S}(u):=|N_S(u)|$.
Let $G\cup G'$ be the union of two vertex-disjoint graphs $G$ and $G'$.
Specially, we use $kG$ to denote the disjoint union of $k$ copies of $G$.
For two disjoint subsets $S,T\subseteq V(G)$,
let $G[S,T]$ be the bipartite subgraph obtained
from $G[S\cup T]$ by deleting all its edges within $S$ and within $T$.
We use $e(S)$ and $e(S,T)$ to denote the numbers of edges in
$G[S]$ and $G[S,T]$, respectively.

\begin{lem}\label{lem3.3}
$d_{L''}(v)\leq \alpha_{\mathbb{H}}$ for each $v\in L'\cup L''$
and $G^*[L'']$ is
$(K_{1,\alpha_{\mathbb{H}}+1}\cup\gamma_{\mathbb{H}}K_1)$-minor-free.
\end{lem}

\begin{proof}
We first show $d_{L''}(v)\leq \alpha_{\mathbb{H}}$ for $v\in L'\cup L''$.
Suppose to the contrary that $d_{L''}(v_0)\geq\alpha_\mathbb{H}+1$
for some $v_0\in L'\cup L''$.
Let $L=\{u_1,\ldots,u_{\gamma_\mathbb{H}}\}$
and $\{w_0,\ldots,w_{\alpha_\mathbb{H}}\}\subseteq N_{L''}(v_0)$.

By equation (\ref{align.09}),
we obtain $|L''|=n-|L|-|L'|\geq\gamma_\mathbb{H}+\alpha_\mathbb{H}+2$
for sufficiently large $n$.
Thus, we can choose $v_1,\ldots,v_{\gamma_{\mathbb{H}}}$
in $L''\setminus\{v_0,w_0,w_1,\ldots,w_{\alpha_\mathbb{H}}\}$.
Note that $G^*[L,L'']\cong K_{|L|,|L''|}.$
Let $G$ be the graph obtained from $G^*$ by
contracting each edge $u_iv_i$ into a new vertex $\overline{u}_i$
for $i\in\{1,\ldots,\gamma_{\mathbb{H}}\}$.
Then, $\{\overline{u}_1,\ldots,\overline{u}_{\gamma_\mathbb{H}}\}$
is a clique in $G$, and $\overline{u}_i\in N_{G}(w_j)$ for
$i\in\{1,\ldots,\gamma_{\mathbb{H}}\}$ and $j\in\{0,\ldots,\alpha_{\mathbb{H}}\}$.

Furthermore, let $G'$ be the graph obtained from $G$
by contracting the edge $v_0w_0$ into a new vertex $\overline{u}_0$.
Recall that $\overline{u}_i\in N_{G}(w_0)$ for $i\in\{1,\ldots,\gamma_{\mathbb{H}}\}$
and $w_j\in N_{G}(v_0)$ for $j\in\{1,\ldots,\alpha_{\mathbb{H}}\}$.
Thus, $\overline{u}_i,w_j\in N_{G'}(\overline{u}_0)$ for
$i\in\{1,\ldots,\gamma_{\mathbb{H}}\}$ and $j\in\{1,\ldots,\alpha_{\mathbb{H}}\}$.
Now, $G'[\{\overline{u}_i,w_j: ~0\leq i\leq\gamma_{\mathbb{H}};
1\leq j\leq\alpha_{\mathbb{H}}\}]$ contains
$B_{\gamma_{\mathbb{H}}+1,\alpha_{\mathbb{H}}}$ as a spanning subgraph.
By Lemma \ref{lem2.3},
$G'$ contains an $H$ minor for some $H\in \mathbb{H}$, and so does $G^*$, leading to a contradiction.
Hence, $d_{L''}(v)\leq\alpha_{\mathbb{H}}$ for each $v\in L'\cup L''$.

Next, suppose that
$G^*[L'']$ contains an $H_0$ minor, where $H_0\cong K_{1,\alpha_{\mathbb{H}}+1}\cup\gamma_{\mathbb{H}}K_1.$
Let $G''$ be the graph obtained from $G^*$ by replacing $G^*[L'']$
with a copy of $H_0$.
Then, $|H_0|=\gamma_\mathbb{H}+\alpha_\mathbb{H}+2$, and there exists $v_0\in V(H_0)$ with $d_{G''}(v_0)=\alpha_\mathbb{H}+1$.
Using a similar argument as above, $G''$ contains an $H$ minor for some $H\in \mathbb{H}$,
and so does $G^*$, a contradiction.
Therefore, the lemma follows.
\end{proof}

\begin{lem}\label{lem3.4}
 $x_v\le\frac{4 x_{u^*}}{100C_\mathbb{H}}$ for each $v\in L'\cup L''$.
\end{lem}

\begin{proof}
Choose an arbitrary $v\in L'\cup L''$.
Then $d_{L}(v)\leq|L|=\gamma_\mathbb{H}$, and
by Lemma \ref{lem3.3}, $d_{L''}(v)\leq\alpha_\mathbb{H}$.
Recall that
$\gamma_\mathbb{H}+\alpha_\mathbb{H}<C_\mathbb{H}$.
Thus, we have
\begin{eqnarray*}
d_{G^*}(v)=d_{L\cup L''}(v)+d_{L'}(v)
\leq(\gamma_\mathbb{H}+\alpha_\mathbb{H})+d_{L'}(v)<
C_\mathbb{H}+d_{L'}(v).
\end{eqnarray*}
Since $\rho^*x_v=\sum_{u\in N_{G^*}(v)}x_u\leq d_{G^*}(v)x_{u^*}$,
we obtain
\begin{eqnarray*}
\sum_{v\in L'}\rho^*x_v
\le\sum_{v\in L'}\big(C_\mathbb{H}+d_{L'}(v)\big)x_{u^*}
=\big(C_\mathbb{H}|L'|+2e(L')\big)x_{u^*},
\end{eqnarray*}
where $e(L')<C_\mathbb{H}|L'|$ by Lemma \ref{lem2.2}.
Combining (\ref{align.09}) gives
$\rho^*\sum_{v\in L'}x_v
\leq3C_\mathbb{H}|L'|x_{u^*}\leq\frac{3\gamma_\mathbb{H}n}{100C_\mathbb{H}}x_{u^*}.$
Again by $d_{L\cup L''}(v)\leq
\gamma_\mathbb{H}+\alpha_\mathbb{H}<C_\mathbb{H}$,
we have
\begin{eqnarray}\label{align.10}
\rho^*x_v=\sum_{u\in N_{G^*}(v)}x_u=\sum_{u\in N_{L\cup L''}(v)}x_u
+\sum_{u\in L'}x_u\le C_\mathbb{H}x_{u^*}+\frac{3\gamma_\mathbb{H}n}{100C_\mathbb{H}\rho^*}x_{u^*}.
\end{eqnarray}
Dividing both sides of (\ref{align.10}) by $\rho^*$
and combining ${\rho^*}^2\geq\gamma_\mathbb{H}(n-\gamma_\mathbb{H})$,
we obtain that $x_v\le\frac{4x_{u^*}}{100C_\mathbb{H}}$
for sufficiently large $n$,
as desired.
\end{proof}

In what follows, we are ready to prove a key lemma, which states that $L'$ is empty.
The proof proceeds by contradiction, employing the absorbing method.
For a finite family $\mathbb{H}$,
the absorbing method directly produces an $\mathbb{H}$-minor-free graph whose defining property is straightforward to check.
When $\mathbb{H}$ is infinite, however, the task becomes subtler:
graphs of sufficiently large order may contain some member of $\mathbb{H}$ as a minor, 
making it difficult to certify that the constructed graph remains $\mathbb{H}$-minor-free.
To enhance readability,
we break the proof of Lemma \ref{lem3.5} into a sequence of claims that
outline the argument and systematically apply the absorbing method,
as follows.
\begin{enumerate}[(i)]\setlength{\itemsep}{0pt}
\item Construct a graph $G'$ such that
$\rho(G')>\rho(G^*)$ and $G'[L,L'\cup L'']$ is a complete bipartite graph.
Then $G'$ admits an $H'$ minor for some $H'\in\mathbb{H}$,
and thus admits an $H'$-partition $\mathcal{V}$.
Using the $H'$-partition $\mathcal{V}$,
we find a maximal linear path $P^*$ of order at least
$\frac{\sqrt{n}}{2|H'|}$.

\item Based on $G'$,
we construct a graph $G''$ by absorbing each vertex of $L'$
into $P^*$; specifically, we replace $P^*$ and $G^*[L']$ with a new maximal linear path $P$ of order $|P^*|+|L'|$.
We then show that $\rho(G'')>\rho(G^*)$, implying that
$G''$ contains an $H''$ minor for some $H''\in\mathbb{H}$, and hence,
it admits an $H''$-partition $\mathcal{V}''$.

\item
Using $\mathcal{V}''$, we construct a graph $G'''$,
by contracting $P$ in $G''$ into a new path
of order $r\leq2|H''|+1$, ensuring that $G'''$ still contains an $H''$ minor.

\item Contract $P^*$ in $G^*-L'$ into a new path
of order $r$ to obtain a graph isomorphic to $G'''$.
Consequently, $G^*$ contains a $G'''$ minor and thus an $H''$ minor,
a contradiction.
\end{enumerate}

\begin{lem}\label{lem3.5}
$L'$ is an empty set.
\end{lem}

\begin{proof}
Suppose to the contrary that $L'\neq\varnothing$.
Let $G'$ be the graph obtained from
$G^*$ by deleting all edges incident to vertices in $L'$
and adding all edges between $L'$ and $L$.

\begin{claim}\label{cl3.1}
Let $\rho':=\rho(G').$ Then $\rho'>\rho^*$.
\end{claim}

\begin{proof}
Since $e(L')\leq C_\mathbb{H}|L'|$,
there exists $v_1\in L'$ with
$d_{L'}(v_1)\leq\frac{2e(L')}{|L'|}\leq2C_\mathbb{H}$.
Set $L'_1:=L'$ and $L'_2:=L_1'\setminus\{v_1\}$.
Then we also have $e(L'_2)\leq C_\mathbb{H}|L'_2|$,
and thus there exists $v_2\in L'_2$ with
$d_{L'_2}(v_2)\leq2C_\mathbb{H}$.
Repeating this step, we obtain a sequence
$L'_1,\ldots,L'_{|L'|}$ such that
$L'_{i+1}=L'_{i}\setminus\{v_i\}$
and $d_{L'_i}(v_i)\leq 2C_\mathbb{H}$ for each $i$.
Now, we can decompose $E(G^*[L'])$ into
$|L'|$ subsets $\{v_iv:~ v\in N_{L'_i}(v_i)\}$,
where $i=1,\ldots,|L'|$.
Therefore,
\begin{eqnarray}\label{align.11}
\rho'-\rho^*\geq {X^*}^T\big(A(G')-A(G^*)\big)X^*
\geq\sum_{i=1}^{|L'|}2x_{v_i}\Big(\sum_{u\in L}x_u
-\!\!\!\sum_{v\in N_{L\cup L''\cup L'_i}(v_i)}\!\!\!x_v\Big).
\end{eqnarray}
Recall that $\gamma_\mathbb{H}+\alpha_\mathbb{H}<C_\mathbb{H}$.
Moreover, by Proposition \ref{pro3.0},
we have $|L|=\gamma_\mathbb{H}$ and
\begin{eqnarray}\label{align.12}
\sum_{u\in L}x_u\geq
\gamma_\mathbb{H}\big(1-\frac{1}{2(10C_\mathbb{H})^2}\big)x_{u^*}\geq
\big(\gamma_\mathbb{H}-\frac{1}{10}\big)x_{u^*}.
\end{eqnarray}

On the other hand, for each $v_i\in L'$,
we have $d_{L}(v_i)\leq|L|-1=\gamma_{\mathbb{H}}-1$ by the definition of $L'$,
$d_{L''}(v_i)\leq\alpha_\mathbb{H}$ by Lemma \ref{lem3.3},
and $d_{L'_i}(v_i)\leq 2C_\mathbb{H}$ by the choice of $v_i$.
Moreover, by Lemma \ref{lem3.4}, we have
$x_v\le\frac{4x_{u^*}}{100C_\mathbb{H}}$ for $v\in L'\cup L''$.
Thus,
$$
\sum_{v\in N_{L\cup L''\cup L'_i}(v_i)}\!\!\!x_v
\leq d_{L}(v_i)x_{u^*}+d_{L''\cup L'_i}(v_i)
\frac{4x_{u^*}}{100C_\mathbb{H}}
\leq(\gamma_\mathbb{H}-1)x_{u^*}+(\alpha_\mathbb{H}+2C_\mathbb{H})\frac{4 x_{u^*}}{100C_\mathbb{H}},
$$
which implies that
$\sum_{v\in N_{L\cup L''\cup L'_i}(v_i)}x_v<\big(\gamma_\mathbb{H}-\frac{1}{10}\big)x_{u^*}$,
as $\alpha_\mathbb{H}+2C_\mathbb{H}<3C_\mathbb{H}$.
Combining this with (\ref{align.11}) and (\ref{align.12}),
we obtain $\rho'\geq\rho^*$,
with equality if and only if $X^*$ is an eigenvector
corresponding to $\rho(G')$
and $x_{v_i}=0$ for each $v_i\in L'$.
Observe that $G'$ is connected.
If $X^*$ is an eigenvector corresponding to $\rho(G')$,
then by the Perron-Frobenius theorem, $X^*$ is positive.
Hence, $\rho'>\rho^*$, as desired.
\end{proof}

In view of Claim \ref{cl3.1} and
the choice of $G^*$,
$G'$ must contain an $H'$ minor for some $H'\in\mathbb{H}$.
Let $\mathcal{V}=(V_1,\dots,V_{|H'|})$ be a minimal $H'$-partition of $G'$.
A set $V_i$ ($i\in\{1,\ldots,|H'|\}$) is called a \emph{good set}
if both $V_i\cap L$ and $V_i\setminus L$ are non-empty.
Since $|L|=\gamma_\mathbb{H}$ and $V_1,\dots,V_{|H'|}$ are vertex-disjoint,
there are at most $\gamma_\mathbb{H}$ good sets in $\mathcal{V}$.
We now present a precise characterization of good sets.

\begin{claim}\label{cl3.2}
Every good set has exactly two vertices.
\end{claim}

\begin{proof}
The definition implies that every good set has at least two vertices.
Suppose, for contradiction, that there exists a good set $V_i$ with $|V_i|\geq3$.
Choose $u\in V_i\cap L$ and $v\in V_i\setminus L$.
Note that $G'[L,L'\cup L'']\cong K_{|L'|,|L'\cup L''|}$.
Thus $L'\cup L''\subseteq N_{G'}(u)$ and $L\subseteq N_{G'}(v)$.
If we contract the edge $uv$ in $G'$ into a new vertex $w$,
then $w$ is a dominating vertex in the resulting graph.
Let $V_i'=\{u,v\}$ and $\mathcal{V}'=(V_1,\dots,V_i',\dots,V_{|H'|})$.
Then, $\mathcal{V}'$ is also an $H'$-partition of $G'$,
contradicting the minimality of $\mathcal{V}$.
Therefore, the claim holds.
\end{proof}

\begin{claim}\label{cl3.3}
For $i\in\{1,\ldots,|H'|\}$,
every induced subgraph $G'[V_i\cap(L'\cup L'')]$ is connected
provided that $V_i\cap(L'\cup L'')\neq\varnothing$.
\end{claim}

\begin{proof}
Suppose to the contrary that
$G'[V_i\cap(L'\cup L'')]$ is not connected for some $i$.
Then $|V_i\cap(L'\cup L'')|\geq2$.
However, $G'[V_i]$ is connected
by the definition of model.
Thus $V_i\cap L\neq\varnothing$.
This implies that $V_i$ is a good set and $|V_i|\geq3$,
contradicting Claim \ref{cl3.2}.
\end{proof}

In what follows, we assume that $\mathcal{V}=(V_1,\dots,V_{|H'|})$
is a minimal $H'$-partition of $G'$, and subject to this,
$|L'\cap(\cup_{i=1}^{|H'|}V_i)|$ is minimized.
Let $G'[L'']$ have $c$ connected components, $G_1,\dots,G_c$.

\begin{claim}\label{cl3.4}
$L''\subseteq\cup_{i=1}^{|H'|}V_i$ and $c\leq|H'|$.
\end{claim}

\begin{proof}
Note that $G^*[L\cup L'']=G'[L\cup L'']$.
If $\cup_{i=1}^{|H'|}V_i\subseteq L\cup L''$,
then $\mathcal{V}$ is an $H'$-partition of $G^*$,
contradicting the fact that $G^*$ is $H'$-minor-free.
Hence, $L'\cap(\cup_{i=1}^{|H'|}V_i)\neq\varnothing$.

Suppose that
there exists $u\in L''\setminus(\cup_{i=1}^{|H'|}V_i)$.
Choose a vertex $v\in L'\cap(\cup_{i=1}^{|H'|}V_i)$.
We may assume without loss of generality that $v\in L'\cap V_1$.
Set $V_1':=(V_1\setminus\{v\})\cup\{u\}$
and $\mathcal{V}':=(V_1',V_2,\dots,V_{|H'|})$.
Note that $N_{G'}(v)=L\subseteq N_{G'}(u)$.
Then, $\mathcal{V}'$ is also an $H'$-partition of $G'$,
contradicting the minimality of $|L'\cap(\cup_{i=1}^{|H'|}V_i)|$.
Thus, $L''\subseteq\cup_{i=1}^{|H'|}V_i$.

In what follows, we prove that $c\leq |H'|$.
On the one hand,
$\cup_{j=1}^{c}V(G_j)=L''\subseteq\cup_{i=1}^{|H'|}V_i.$
On the other hand, by Claim \ref{cl3.3},
we can see that for each $V_i$ (where $1\leq i\leq |H'|$),
there exists at most one $G_j$ with $V(G_j)\cap V_i\neq\varnothing$.
Therefore, we have $c\leq |H'|$, as desired.
\end{proof}

\begin{claim}\label{cl3.5}
Let $|G_1|=\max_{1\leq j\leq c}|G_j|$.
Then, $e(G_1)\leq|G_1|+\frac16\sqrt{n}.$
\end{claim}

\begin{proof}
It follows from (\ref{align.09}) that $|L''|=n-|L|-|L'|\geq\frac n2$.
Combining Claim \ref{cl3.4},
we obtain
\begin{eqnarray}\label{align.13}
|G_1|\geq\frac1c|L''|\geq\frac1{|H'|}|L''|\geq\frac{n}{2|H'|}.
\end{eqnarray}
This implies that $|G_1|\geq\gamma_{H^*}+\alpha_{H^*}+3$,
where $H^*$ is minimal with respect to $\mathbb{H}$.
Taking a spanning tree of $G_1$.
By successively deleting leaves,
we obtain a $\gamma_{H^*}$-subset $S\subseteq V(G_1)$ such that $G_1-S$ remains connected.
Observe that $G^*[L'']=G'[L'']$. Set $t:=\alpha_{H^*}+1.$
By Lemma \ref{lem3.3},
$G_1$ is $(K_{1,t}\cup\gamma_{H^*}K_1)$-minor-free.
Thus, $G_1-S$ is $K_{1,t}$-minor-free.
Since $G_1-S$ is a connected graph of order at least $t+2$,
it must contain a $K_{1,2}$,
implying $t\geq3$.
By Lemma \ref{lem3.1}, we obtain
$e(G_1-S)\leq\binom{t}{2}+|G_1-S|-t$.
Moreover, Lemma \ref{lem3.3} also gives $d_{L''}(v)\leq\alpha_{H^*}$
for each $v\in L''$.
Hence,
$$e(G_1)\leq e(G_1-S)+\sum_{v\in S}d_{L''}(v)
\leq\frac12t(t-1)+|G_1|-|S|-t+|S|\alpha_{H^*}.$$
It follows that $e(G_1)\leq|G_1|+\frac16\sqrt{n}$ for sufficiently large $n$.
\end{proof}

\begin{claim}\label{cl3.6}
Let $U_1$ be the set of vertices of degree one in $G_1$.
Then, $|U_1|\leq |H'|$.
\end{claim}

\begin{proof}
We first show that
$|V_i\cap U_1|\leq1$ for each $i\in\{1,\ldots,|H'|\}$.
Suppose, to the contrary, that $|V_i\cap U_1|\geq2$ for some $i$.
Note that $V_i\cap U_1\subseteq V_i\setminus L$.
Thus, $|V_i\setminus L|\geq2$.

Now, if $V_i\cap L\neq\varnothing$, then $V_i$ is a good set,
and by Claim \ref{cl3.2}, we have $|V_i|=2$, a contradiction.
Hence, $V_i\cap L=\varnothing$, meaning that
$V_i\subseteq L'\cup L''$.

Since $G'[V_i]$ is connected
and $|V_i|\geq|V_i\cap U_1|\geq2$,
there exist $u_0\in V_i\cap U_1$ and $w_0\in N_{V_i}(u_0)$.
Recall that $L\subseteq N_{G'}(v)$ for each $v\in L'\cup L''$.
Consequently, $L\subseteq N_{G'}(u_0)\cap N_{G'}(w_0)$,
as $u_0,w_0\in V_i\subseteq L'\cup L''$.
Since $u_0\in U_1$ and $G_1$ is a component of $G'[L'\cup L'']$,
we know that $w_0$ is the unique neighbor of $u_0$ in $L'\cup L''$,
and thus
$N_{G'}(u_0)\setminus\{w_0\}\subseteq N_{G'}(w_0)\setminus\{u_0\}.$
Hence, $(V_1,\dots,V_i\setminus\{u_0\},\dots,V_{|H'|})$
is also an $H'$-partition of $G'$,
contradicting the minimality of $\mathcal{V}$.
Hence, $|V_i\cap U_1|\leq1$ for each $i$.

Note that $U_1\subseteq V(G_1)\subseteq L''$.
Moreover, by Claim \ref{cl3.4}, we have $L''\subseteq\cup_{i=1}^{|H'|}V_i$.
Thus, $U_1\subseteq\cup_{i=1}^{|H'|}V_i$.
Since $|V_i\cap U_1|\leq1$ for each $i\in\{1,\ldots,|H'|\}$,
we obtain $|U_1|\leq |H'|$.
\end{proof}

\begin{claim}\label{cl3.7}
$G_1$ admits a maximal linear path $P^*$ of order at least
$\frac{\sqrt{n}}{2|H'|}$.
\end{claim}

\begin{proof}
Let $U_2=\{v\in V(G_1): d_{G_1}(v)=2\}$ and $U_3=V(G_1)\setminus(U_1\cup U_2).$
Then,
$3|G_1|-2|U_1|-|U_2|=|U_1|+2|U_2|+3|U_3|
\leq2e(G_1),$
which yields that $e(G_1)-|U_2|\leq 3\big(e(G_1)-|G_1|\big)+2|U_1|$.
Combining Claims \ref{cl3.5} and \ref{cl3.6} gives that
\begin{eqnarray}\label{align.14}
e(G_1)-|U_2|\leq
\frac12\sqrt{n}+2|H'|\leq\sqrt{n}.
\end{eqnarray}

Assume now that there are $\phi(G_1)$ maximal linear paths in $G_1$.
If $G_1$ itself is a cycle,
then by Definition \ref{de3.1}, $\phi(G_1)=1$.
Otherwise,
by Proposition \ref{pro3.1},
the two ends of every maximal linear path contributes exactly two to the sum
$\sum_{v\in V(G_1)\setminus U_2}d_{G_1}(v)$.
Thus, $2\phi(G_1)=\sum_{v\in V(G_1)\setminus U_2}d_{G_1}(v).$
Combining inequality (\ref{align.14}) gives
$$\phi(G_1)=\frac{1}{2}\sum_{v\in V(G_1)\setminus U_2}\!\!\!d_{G_1}(v)
=e(G_1)-|U_2|\leq\sqrt{n}.$$
By Proposition \ref{pro3.2},
$G_1$ admits an edge-decomposition of its maximal linear paths.
Furthermore, by inequality (\ref{align.13}) and the pigeonhole principle, there exists a maximal linear path $P^*$ such that
$|P^*|\geq\frac{e(G_1)}{\phi(G_1)}+1
\geq\frac{|G_1|}{\phi(G_1)}\geq\frac{\sqrt{n}}{2|H'|},$
as desired.
\end{proof}

Observe that $V(P^*)\subseteq L''$ and $L'$ is an independent set in $G'$.
We now absorb every vertex in $L'$ into $P^*$.
Write $P^*=w_1w_2\ldots w_a$ and $L'=\{v_1,v_2,\ldots,v_b\}$.
Let $G''$ be the graph obtained from $G'$ by replacing the edge $w_1w_2$
with the path $w_1v_1v_2\ldots v_bw_2$.

\begin{claim}\label{cl3.8}
Let $\rho'':=\rho(G'')$. Then, $\rho''>\rho^*$.
\end{claim}

\begin{proof}
By Claim \ref{cl3.1}, $\rho(G')=\rho'>\rho^*$.
Then it suffices to show $\rho''\geq\rho'.$
Since $G'$ is connected,
by the Perron-Frobenius theorem, there exists a positive unit eigenvector
$Y=(y_1,\ldots,y_n)^T$ corresponding to $\rho(G')$.
Set $\sigma_L:=\sum_{u\in L}y_u$
and $y_{L''}^*:=\max_{w\in L''}y_w$.
For every $v_i\in L'$, since $N_{G'}(v_i)=L$,
we have $\rho'y_{v_i}=\sigma_L$.
Consequently, $y_{v_1}=y_{v_b}=\frac{\sigma_L}{\rho'}.$
Moreover, by Lemma \ref{lem3.3},
$d_{L''}(w)\leq\alpha_{\mathbb{H}}$ for each $w\in L''$.
Thus, $\rho'y_{L''}^*\leq\sigma_L+\alpha_{\mathbb{H}}y_{L''}^*$,
which yields $y_{L''}^*\leq\frac{\sigma_L}{\rho'-\alpha_{\mathbb{H}}}.$
By Lemma \ref{lem2.4},
$\rho^*\geq\sqrt{\gamma_{\mathbb{H}}(n-\gamma_{\mathbb{H}})}$,
and thus $\rho'>\sqrt{\gamma_{\mathbb{H}}(n-\gamma_{\mathbb{H}})}\geq2\alpha_{\mathbb{H}}$.
Combining these inequalities yields
\begin{eqnarray}\label{align.15}
\max\{y_{w_1},y_{w_2}\}\leq y_{L''}^*\leq\frac{\sigma_L}{\rho'-\alpha_{\mathbb{H}}}
\leq\frac{2\sigma_L}{\rho'}=2y_{v_1}=2y_{v_b}.
\end{eqnarray}
On the other hand, one can see that
$$\rho''-\rho'
\geq 2(y_{w_1}y_{v_1}+y_{w_2}y_{v_b}-y_{w_1}y_{w_2})
=y_{w_1}(2y_{v_1}-y_{w_2})+y_{w_2}(2y_{v_b}-y_{w_1}).$$
Combining (\ref{align.15}), we have $\rho''\geq\rho',$
and so $\rho''>\rho^*$.
\end{proof}

We are now ready to complete the proof of Lemma \ref{lem3.5}.
Note that $G''[L'\cup L'']$ has a maximal linear path
$P=w_1v_1v_2\ldots v_bw_2\ldots w_{a-1}w_a$.
In view of
Claim \ref{cl3.8} and the definition of $G^*$,
$G''$ contains an $H''$ minor for some $H''\in\mathbb{H}$.
Let $\mathcal{V''}=(V_1,\dots,V_{|H''|})$ be a minimal $H''$-partition of $G''$.
Applying Claim \ref{cl3.3} to $G''$ and $\mathcal{V''}$,
we know that for every $i\in\{1,\ldots,|H''|\}$,
either $V_i\cap(L'\cup L'')=\varnothing$ or
$G''[V_i\cap(L'\cup L'')]$ is connected.
This implies that exactly one of the following holds:
\begin{itemize}\setlength{\itemsep}{0pt}
\item $V_i\cap V(P)=\varnothing$;
\item $G''[V_i\cap V(P)]$ is a subpath of $P$; 
\item $G''[V_i\cap V(P)]$ consists of two subpaths $P',P''$ such that
$w_1\in V(P')$ and $w_a\in V(P'')$.
\end{itemize}
Now, let $P(i)=G''[V_i\cap V(P)]$, $i=1,\ldots,|H''|.$
Then $P(1),\ldots,P(|H''|)$ partition $P$ into $r$ subpaths,
where $r\leq 2|H''|+1$.
Let $G'''$ denote the graph obtained from $G''$
by contracting each of these $r$ subpaths into a new vertex.
Then by the definition of a model, $G'''$ still contains an $H''$ minor.

On the other hand, recall that $G^*[L'']=G'[L'']$ and that
$P^*\subseteq G'[L'']$.
Hence, $P^*$ is also a maximal linear path in $G^*[L'']$.
By Claim \ref{cl3.7}, $|P^*|\geq\frac{\sqrt{n}}{2|H'|}\geq2|H''|+1$.
Contracting the path $P^*$ in $G^*-L'$ into a new path of order $r$ yields a
graph isomorphic to $G'''$.
Hence, $G^*$ contains a $G'''$ minor and thus an $H''$ minor,
a contradiction.
Therefore, $L'=\varnothing$,
completing the proof.
\end{proof}

In the following, we complete the proof of Theorem \ref{thm1.1}.

\begin{proof}
We begin by proving that $G^*$ contains a spanning subgraph $B_{\gamma_\mathbb{H},n-\gamma_\mathbb{H}}$.
By Proposition \ref{pro3.0}, $|L|=\gamma_\mathbb{H}$,
and by Lemma \ref{lem3.5}, we have $V(G^*)=L\cup L''$.
It is then sufficient to show that $G^*[L]\cong K_{\gamma_\mathbb{H}}$.
Suppose, for the sake of contradiction, that there exist two non-adjacent vertices $u_1,u_2\in L$.
Since $|L''|=n-|L|>\max_{H\in\mathbb{H}}(|H|+1)$,
we can choose a subset $L'''\subset L''$
with $|L'''|=\max_{H\in\mathbb{H}}(|H|+1)$.
Let $L'''':=L''\setminus L'''$.
By Lemma \ref{lem3.3}, $d_{L''}(w)\leq\alpha_{\mathbb{H}}$ for each $w\in L''$.
It follows that
$$e(L'')-e(L'''')\leq\sum_{v\in L'''}d_{L''}(v)\leq
\max_{H\in\mathbb{H}}(|H|+1)\alpha_{\mathbb{H}}
\leq\sqrt{n}.$$
Next, let $x_{L''}^*:=\max_{w\in L''}x_w$.
Then, $\rho^*x_{L''}^*\leq\sum_{u\in L}x_u+\alpha_{\mathbb{H}}x_{L''}^*,$
where $\rho^*\geq\sqrt{\gamma_\mathbb{H}(n-\gamma_\mathbb{H})}.$
Hence, $$x_{L''}^*\leq\frac{\sum_{u\in L}x_u}{\rho^*-\alpha_{\mathbb{H}}}
\leq\frac{\gamma_\mathbb{H}x_{u^*}}{\rho^*-\alpha_{\mathbb{H}}}
\leq\sqrt{\frac{2\gamma_\mathbb{H}}{n}}x_{u^*}.$$

Let $E=E(L'')\setminus E(L'''')$, and let $G$ be the graph obtained from
$G^*$ by deleting all edges in $E$ and adding an edge $u_1u_2$.
By Proposition \ref{pro3.0},
$\min\{x_{u_1},x_{u_2}\}\geq(1-\frac{1}{2(10C_\mathbb{H})^2})x_{u^*}$,
and thus
\begin{eqnarray*}
\rho(G)-\rho(G^*)
\geq 2x_{u_1}x_{u_2}-\!\!\!\sum_{w_1w_2\in E}\!\!\!2x_{w_1}x_{w_2}
\geq 2\Big(\big(1-\frac{1}{2(10C_\mathbb{H})^2}\big)^2
-\sqrt{n}\cdot\frac{2\gamma_\mathbb{H}}{n}\Big)x_{u^*}^2
> 0,
\end{eqnarray*}
which implies that $G$ contains an $H_0$ minor for some $H_0\in \mathbb{H}$.

Let $\mathcal{V}=(V_1,\dots,V_{|H_0|})$ be a minimal $H_0$-partition of $G$.
Since $N_G(v)=L$ for each vertex $v\in L'''$,
applying Claim \ref{cl3.3} to $G$ and $\mathcal{V}$ yields $|V_i\cap L'''|\leq1$ for $i=1,\ldots,|H_0|$.
Since $|L'''|=\max_{H\in\mathbb{H}}(|H|+1)\geq|H_0|+1$,
there exists a vertex
$v\in L'''\setminus\cup_{i=1}^{|H_0|}V_i$.
Consequently,
$\mathcal{V}=(V_1,\dots,V_{|H_0|})$ is also an $H_0$-partition of $G-\{v\}$,
implying that $G-\{v\}$ contains an $H_0$ minor.

Now, let $G^*_{u_1v}$ be the graph obtained from $G^*$ by contracting the edge $u_1v$
into a new vertex $\overline{u}_1$. Then $\overline{u}_1u_2\in E(G^*_{u_1v})$,
and thus $G-\{v\}$ is isomorphic to some subgraph of $G^*_{u_1v}$.
This implies that $G^*_{u_1v}$ also contains an $H_0$ minor.
Correspondingly, $G^*$ contains an $H_0$ minor, a contradiction.
Therefore, $G^*$ contains a spanning subgraph $B_{\gamma_\mathbb{H},n-\gamma_\mathbb{H}}$.

We now derive the lower bound for $\rho^*=spex(n,\mathbb{H}_{minor})$.
A straightforward calculation gives $\rho(B_{\gamma_\mathbb{H},n-\gamma_\mathbb{H}})
=\frac{\gamma_\mathbb{H}-1}{2}+\sqrt{\gamma_\mathbb{H} n-\frac{3\gamma_\mathbb{H}^2+2\gamma_\mathbb{H}-1}{4}}$.
It follows that
   $$\rho^*\geq  \rho(B_{\gamma_\mathbb{H},n-\gamma_\mathbb{H}})
\geq \sqrt{\gamma_\mathbb{H} n}+\frac{\gamma_\mathbb{H}-1}{2}+O\big(\frac{1}{\sqrt{n}}\big).$$

It remains to obtain the upper bound of $\rho^*$.
Recall that $\rho^*x_{L''}^*\leq\sum_{u\in L}x_u+\alpha_{\mathbb{H}}x_{L''}^*$ and $|L|=\gamma_\mathbb{H}$.
By symmetry, $x_u=x_{u^*}$ for each $u\in L$.
Therefore, $\rho^*x_{u^*}\leq(\gamma_\mathbb{H}-1)x_{u^*}+(n-\gamma_\mathbb{H})x_{L''}^*$
and $\rho^*x_{L''}^*\leq\gamma_\mathbb{H}x_{u^*}+\alpha_{\mathbb{H}}x_{L''}^*.$
Combining these two inequalities gives
 $(\rho^*\!-\!\gamma_\mathbb{H}\!+\!1)(\rho^*\!-\!\alpha_\mathbb{H})\leq(n\!-\!\gamma_\mathbb{H})\gamma_\mathbb{H}$.
It follows that $\rho^*
\leq \sqrt{\gamma_\mathbb{H} n}+\frac{\alpha_\mathbb{H}+\gamma_\mathbb{H}-1}{2}+O(\frac{1}{\sqrt{n}})$,
completing the proof of Theorem \ref{thm1.1}.
\end{proof}

Having established Theorem \ref{thm1.1},
we now proceed to characterize $G^*-L$.
Let us recall some notations and terminologies.
For a member $H\in\mathbb{H},$
$\Gamma_s^*(H)$ denotes the family of $s$-vertex
irreducible induced subgraphs of $H$, and
$\Gamma(\mathbb{H})=\bigcup_{H\in\mathbb{H}}\Gamma_{|H|-\gamma_\mathbb{H}}^*(H),$
where $\gamma_\mathbb{H}=\min_{H\in\mathbb{H}}\gamma_H.$

\begin{lem}\label{lem3.6}
Let $G$ be a graph with a set $L$ of $\gamma_\mathbb{H}$ dominating vertices.
Then, $G$ is $\mathbb{H}$-minor-free if and only if $G-L$
is $\Gamma(\mathbb{H})$-minor-free.
\end{lem}

\begin{proof}
Firstly, assume that $G-L$ contains an $H_0$ minor for some
$H_0\in\Gamma(\mathbb{H}).$
Then, there exists an $H\in\mathbb{H}$ such that $H_0$ is
an $(|H|-\gamma_\mathbb{H})$-vertex induced subgraph of $H$.
Combining this $H_0$ minor with $\gamma_\mathbb{H}$ dominating vertices in $L$,
we obtain an $H$ minor in $G$.

Conversely, assume that $G$ contains an $H$ minor for some
$H\in\mathbb{H}.$
By Definition \ref{de3.2},
$G$ has an $H$-partition $\mathcal{V}=(V_1,\dots,V_{|H|})$.
We may assume that $\mathcal{V}$ is a minimal $H$-partition,
and subject to this minimality, $|L\cap(\cup_{i=1}^{|H|}V_i)|$ is maximized.
Since $L$ is a set of $\gamma_\mathbb{H}$ dominating vertices, 
there exist exactly $\gamma_\mathbb{H}$ members of $\mathcal{V}$,
say $V_1,\ldots,V_{\gamma_\mathbb{H}}$,
such that $|V_i|=|L\cap V_i|=1$ for $i\in\{1,\ldots,\gamma_\mathbb{H}\}$.
Consequently, $\cup_{i=\gamma_\mathbb{H}+1}^{|H|}V_i\subseteq V(G)\setminus L,$
which implies that $G-L$ contains an $H_0$ minor for some $H_0\in \Gamma_{|H|-\gamma_\mathbb{H}}(H)$.
Therefore, $G-L$ also contains an $H_0^*$ minor for some $H_0^*\in \Gamma^*_{|H|-\gamma_\mathbb{H}}(H)$.
\end{proof}

In the following, we provide the proof of Theorem \ref{thm1.2}.

\begin{proof}
By Theorem \ref{thm1.1}, $G^*$ has $\gamma_\mathbb{H}$ dominating vertices.
Thus, $G^*$ is connected, and adding an arbitrary edge within its independent set strictly increases
the spectral radius.
Furthermore, by Lemma \ref{lem3.6},
$G^*-L$ is $\Gamma(\mathbb{H})$-minor-free.
Therefore, $G^*-L$ is $\Gamma(\mathbb{H})$-minor-saturated.

In Particular, if $\mathbb{H}=\{H\}$,
then $|H|-\gamma_H=\alpha_H+1.$
Therefore, $\Gamma(\mathbb{H})=\Gamma_{\alpha_{H}+1}^*(H)$,
and hence, $G^*-L$ is $\Gamma_{\alpha_{H}+1}^*(H)$-minor-saturated.
This completes the proof of Theorem \ref{thm1.2}.
\end{proof}

Set $L'':=V(G^*)\setminus L$.
A subset $R$ of $L''$ is called a \emph{component subset},
if $G^*[R]$ consists of one or more connected components of $G^*[L'']$.
Such an $R$ is further said to be \emph{small}, if $|R|\leq C$ for some constant $C$.
To prove Theorem \ref{thm1.3}, we will need four additional lemmas. 
The proofs involve detailed constructions, combinatorial analysis, 
and the application of spectral methods.

\begin{lem}\label{lem3.7}
If $R$ is a small component subset of $L''$,
then $e(R)=ex\big(|R|,\Gamma(\mathbb{H})_{minor}\big)$.
\end{lem}

\begin{proof}
Let $H$ be a minimal member in $\mathbb{H}$.
Recall that $\gamma_{H}=\gamma_\mathbb{H}$ and $\alpha_{H}=\alpha_\mathbb{H}$.
Then $|H|-\gamma_\mathbb{H}=\alpha_\mathbb{H}+1$, and thus
$\Gamma_{\alpha_{\mathbb{H}}+1}^*(H)\subseteq \Gamma(\mathbb{H}).$
By Lemma \ref{lem3.6},
$G^*[L'']$ is $\Gamma_{\alpha_{\mathbb{H}}+1}^*(H)$-minor-free.
Now, choose an $(\alpha_{\mathbb{H}}+1)$-subset $S$ of $V(H)$
such that it contains $\alpha_{\mathbb{H}}$ independent vertices of $H$.
Then, $G^*[L'']$ is $H$-minor-free.
Since $H[S]$ is a subgraph of $K_{1,\alpha_{\mathbb{H}}}$,
$G^*[L'']$ is also $K_{1,\alpha_{\mathbb{H}}}$-minor-free.

Define $x^*_{L''}:=\max_{w\in L''}x_w$
and $\sum_{u\in L}x_u:=\sigma_L$.
Clearly, $x_w\geq\frac{\sigma_L}{\rho^*}$ for each $w\in L''$.
Since $G^*[L'']$ is $K_{1,\alpha_{\mathbb{H}}}$-minor-free,
we have $\rho^*x^*_{L''}\leq \sigma_L+(\alpha_{\mathbb{H}}-1)x^*_{L''},$
which gives
$x^*_{L''}\leq\frac{\sigma_L}{\rho^*-\alpha_{\mathbb{H}}+1}.$

Now, suppose to the contrary that there exists a small component subset
$R$ of $L''$ such that $e(R)<ex\big(|R|,\Gamma(\mathbb{H})_{minor}\big)$.
Then, we can find a $\Gamma(\mathbb{H})$-minor-free graph $G_i''$
on the vertex set $R$ with at least $e(R)+1$ edges.
Let $G$ be the graph obtained from $G^*$ by replacing $E(G^*[R])$ with $E(G_i'')$.
Since $\Gamma(\mathbb{H})$ is a connected family,
$G[L'']$ remains $\Gamma(\mathbb{H})$-minor-free.
By Lemma \ref{lem3.6}, $G$ is also $\mathbb{H}$-minor-free.
Let $\rho=\rho(G)$. We can deduce that
\begin{eqnarray}\label{align.21}
\frac12(\rho-\rho^*)\geq\!\!\!\sum\limits_{uv\in E(G_i'')}\!\!\!x_ux_v
-\!\!\!\sum\limits_{uv\in E(G^*[R])}\!\!\!x_ux_v\geq \frac{e(G_i'')\sigma_L^2}{{\rho^*}^2}
-\frac{e(R)\sigma_L^2}{(\rho^*-\alpha_{\mathbb{H}}+1)^2}.
\end{eqnarray}

Since $R$ is small, $e(R)$ is also bounded by a constant.
Recall that $e(G_i'')\geq e(R)+1$ and $\rho^*\geq\sqrt{\gamma_\mathbb{H}(n-\gamma_\mathbb{H})}$.
Clearly, we have $\rho(G)>\rho^*$
for sufficiently large $n$, a contradiction.
\end{proof}

Let $\mathbb{G}$ be the family
of connected $\Gamma(\mathbb{H})$-minor-free graphs
on at most $n-\gamma_\mathbb{H}$ vertices.
The proof of Lemma \ref{lem3.7} implies that
every member in $\mathbb{G}$ is $K_{1,\alpha_\mathbb{H}}$-minor-free.
Given a member $G_i\in\mathbb{G}$,
denote by $d_i$ its average degree.
We say that $G_i$ is \emph{small} if $|G_i|\leq c$ for some constant $c$
(i.e., $|G_i|$ is independent of $n$).
Let $G_0$ be a member of $\mathbb{G}$ with $d_0=\max_{G_i\in\mathbb{G}}d_i$,
and let $G_1,\ldots,G_s$ be all the non-isomorphic components
in $G^*-L$. We assume that $m(G_1)\geq\cdots\geq m(G_s)$,
where $m(G_i)$ is the number of copies of $G_i$ in $G^*-L$.
By Lemma \ref{lem3.6}, we also have
$G_i\in\mathbb{G}$ for $i\in\{1,\ldots,s\}$.

\begin{lem}\label{lem3.8}
If $G_i$, $i=0,\ldots,s$, are all small,
then $d_1=d_0$ and
$m(G_i)<|G_1|$ for $G_i$ with $d_i<d_0$.
\end{lem}

\begin{proof}
By the choice of $G_0$,
we know that $d_i\leq d_0$ for every $i\in\{1\ldots,s\}$.
Since $\max_{1\leq i\leq s}|G_i|=c$ for some constant integer $c$,
we have $s\leq\sum_{k=1}^c2^{k \choose 2}$, i.e., $s$ is constant.
However, $|L''|$ is sufficiently large,
it follows that $m(G_1)$ is a function of $n$.

We first prove that $d_1=d_0$.
Suppose, for the sake of contradiction, that $d_1<d_0$.
We define $G$ to be
the graph obtain from $G^*$ by replacing $|G_0|$ copies of $G_1$ with
$|G_1|$ copies of $G_0$ in $G^*-L$.
The component subset $V(|G_0|G_1)$ is small,
and $G[L'']$ remains $\Gamma(\mathbb{H})$-minor-free.
However, $|G_0|e(G_1)-|G_1|e(G_0)=\frac12|G_0||G_1|(d_1-d_0)<0$,
contradicting Lemma \ref{lem3.7}.
Thus, we have $d_1=d_0$.

Now, suppose that there exists a component $G_i$
with $d_i<d_0$ and $m(G_i)\geq|G_1|$.
We define $G$ to be
a new graph obtain from $G^*$ by replacing $|G_1|$ copies of $G_i$ with
$|G_i|$ copies of $G_1$ in $G^*[L'']$.
Clearly, we have $|G_1|e(G_i)-|G_i|e(G_1)=\frac12|G_1||G_i|(d_i-d_0)<0$,
which leads to a contradiction, similar to the previous case.
Hence, the lemma holds.
\end{proof}

Let $G^\star$ be a graph with $V(G^\star)=V(G^*)$ and the same set $L$ of dominating vertices,
such that $G^\star-L\in EX(n-|L|,\Gamma(\mathbb{H})_{minor})$.
Furthermore, assume that
$G_{i_1},\ldots,G_{i_t}$ are all the non-isomorphic components
in $G^\star-L$. Clearly, $G_{i_j}\in\mathbb{G}$ for every $j\in\{1,\ldots,t\}$.

\begin{lem}\label{lem3.9}
If both $\{G_i: i=0,1,\ldots,s\}$
and $\{G_{i_j}: j=1,\ldots,t\}$ are small graph families,
then $G^*-L$ also belongs to $EX(n-|L|,\Gamma(\mathbb{H})_{minor})$.
\end{lem}

\begin{proof}
Suppose, for contradiction, that $e(G^*-L)<e(G^\star-L)$.
Since $\max_{1\leq j\leq t}|G_{i_j}|$ is constant,
so is $t$.
The proof of Lemma \ref{lem3.8} implies that
$m(G_{i_j})<|G_1|$ for every $G_{i_j}$ with $d_{i_j}<d_0$.
Recall that $m(G_i)$ denotes the number of copies of $G_i$ in $G^*-L$.
Now, define $m'(G_{i_j})$ as the number of copies of $G_{i_j}$ in $G^\star-L$.
Let ${m_i}\equiv m(G_i) \pmod{|G_1|}$
for $i\in\{1,\ldots,s\}$ and ${m'_{i_j}}\equiv m'(G_{i_j}) \pmod{|G_1|}$ for $j\in\{1,\ldots,t\}$.
We can observe that
${m_i}=m(G_i)$ if $d_i<d_0$,
and ${m'_{i_j}}=m(G'_{i_j})$ if $d_{i_j}<d_0$.

Now, let $G'_1=G^*-L-\cup_{i=1}^sV(m_iG_i)$
and $G'_2=G^\star-L-\cup_{j=1}^tV(m'_{i_j}G_{i_j})$ for simplicity.
Then, both $|G'_1|$ and
$|G'_2|$ are divisible by $|G_1|$, which leads to:
\begin{eqnarray}\label{align.22}
|\cup_{j=1}^tV(m'_{i_j}G_{i_j})|-|\cup_{i=1}^sV(m_iG_i)|
=|G'_1|-|G'_2|=r|G_1|
\end{eqnarray}
for some integer $r$.
Since both $|\cup_{j=1}^tV(m'_{i_j}G_{i_j})|$ and
$|\cup_{i=1}^sV(m_iG_i)|$ are finite, it follows that $r$ is constant.
Observe that every component in $G'_1$ and $G'_2$
has average degree $d_0$.
Combining (\ref{align.22}), we obtain
$e(G'_1)-e(G'_2)=\frac12d_0(|G'_1|-|G'_2|)=re(G_1)$.
Consequently, we have
\begin{eqnarray}\label{align.23}
\sum_{j=1}^te(m'_{i_j}G_{i_j})-\sum_{i=1}^se(m_iG_i)
=e(G^\star-L)-e(G^*-L)+re(G_1)>re(G_1).
\end{eqnarray}

Recall that $m(G_1)$ is a function of $n$.
If $r\geq0$, we define $R=(\cup_{i=1}^sV(m_iG_i))\cup V(rG_1)$.
Then, $R$ is a small component subset of $L''$,
and $|R|=|\cup_{j=1}^tV(m'_{i_j}G_{i_j})|$
by (\ref{align.22}).
However, from (\ref{align.23}), we have 
$e(R)<\sum_{j=1}^te(m'_{i_j}G_{i_j})$,
which contradicts Lemma \ref{lem3.7}.

If $r<0$, we define $R=\cup_{i=1}^sV(m_iG_i)$.
Then, $R$ is still a small component subset of $L''$,
and from (\ref{align.22}), we have
$|R|=|\cup_{j=1}^tV(m'_{i_j}G_{i_j})|+|V((-r)G_1)|$.
However, by (\ref{align.23}),
we obtain $e(R)<\sum_{j=1}^te(m'_{i_j}G_{i_j})+(-r)e(G_1)$,
which also leads to a contradiction.
\end{proof}

\begin{lem}\label{lem3.10}
If there exists a member in $\mathbb{G}$ that contains a bicyclic subgraph,
then both $\{G_i: i=0,1,\ldots,s\}$
and $\{G_{i_j}: j=1,\ldots,t\}$ consist entirely of small graphs.
\end{lem}

\begin{proof}
Suppose that some member in $\mathbb{G}$ contains a bicyclic subgraph.
Then, it has an $H_0$ minor
where $H_0\in\{K_1\nabla 2P_2,K_1\nabla P_3\}$,
and thus a $K_{1,3}$ minor.
Recall that every member in $\mathbb{G}$ is $K_{1,\alpha_\mathbb{H}}$-minor-free.
Hence, we conclude that $\alpha_{\mathbb{H}}\geq4$.

Let $\mathbb{G}'$ be the subset of $\mathbb{G}$
excluding all small members.
If $\mathbb{G}'$ is empty, we are done.
Next, assume that $\mathbb{G}'$ is non-empty, and let
$G_{i_0}\in \mathbb{G}'$.
Since $G_{i_0}$ is $K_{1,\alpha_\mathbb{H}}$-minor-free,
by Lemma \ref{lem3.1}, we have $e(G_{i_0})\leq {\alpha_{\mathbb{H}}\choose 2}+|G_{i_0}|-\alpha_{\mathbb{H}}.$
Thus, $d_{i_0}=2e(G_{i_0})/|G_{i_0}|\rightarrow 2$ as $|G_{i_0}|\rightarrow \infty$,
and since $|G_{i_0}|$ is sufficiently large, we have $d_{i_0}<2.2$.
We now conclude that $d_i<2.2$ for any $G_i\in \mathbb{G}'$.

Recall that some member in $\mathbb{G}$ contains an $H_0$ minor.
Then, $H_0$ is $\Gamma(\mathbb{H})$-minor-free.
A simple calculation gives 
$2e(H_0)/|H_0|\geq2.4$.
By the definition of $G_0$,
we have $d_0\geq2e(H_0)/|H_0|\geq2.4$.
Since $d_i<2.2$ for every $G_i\in \mathbb{G}'$,
it is clear that $G_0\notin\mathbb{G}'$.

Suppose that $G_{i_j}\in\mathbb{G}'$ for some $j\in\{1,\ldots,t\}$.
Then, $d_{i_j}<2.2$ and $|G_{i_j}|=a|H_0|+b$,
where $0\leq b<|H_0|$ and $a$ depends on $n$.
Hence, $e(G_{i_j})<1.1(a|H_0|+b).$
Now, define $G''_{i_j}=aH_0\cup bK_1$.
Then, $G''_{i_j}$ is $\Gamma(\mathbb{H})$-minor-free,
and $e(G''_{i_j})=ae(H_0)\geq1.2a|H_0|.$
Thus, $e(G''_{i_j})>e(G_{i_j}),$
contradicting the fact that
$G^\star[L'']\in EX(n-|L|,\Gamma(\mathbb{H})_{minor})$.
Hence, $G_{i_j}\notin\mathbb{G}'$ for any $j\in\{1,\ldots,t\}$.

Finally, suppose $G_i\in\mathbb{G}'$
for some $i\in\{1,\ldots,s\}$.
Then, $d_{i}<2.2$ and $|G_{i}|=a'|H_0|+b'$,
where $0\leq b'<|H_0|$ and $a'$ depends on $n$.
Thus, $e(G_i)<1.1(a'|H_0|+b').$
Now define $G''_i=a'H_0\cup b'K_1$.
Then $G''_i$ is $\Gamma(\mathbb{H})$-minor-free,
and $e(G''_i)=a'e(H_0)\geq1.2a'|H_0|.$
It is easy to see that $e(G''_i)>1.05e(G_i).$
Let $G$ be the graph obtained from $G^*$
by replacing $E(G_i)$ with $E(G_i'')$.
By Lemma \ref{lem3.6}, $G$ is also $\mathbb{H}$-minor-free.
Choosing $R=V(G_i)$ in inequality (\ref{align.21}),
we obtain
\begin{eqnarray*}
\frac12\big(\rho(G)-\rho^*\big)\geq\!\!\!\sum\limits_{uv\in E(G_i'')}\!\!\!x_ux_v
-\!\!\!\sum\limits_{uv\in E(G_i)}\!\!\!x_ux_v\geq \frac{e(G_i'')\sigma_L^2}{{\rho^*}^2}
-\frac{e(G_i)\sigma_L^2}{(\rho^*-\alpha_{\mathbb{H}}+1)^2}.
\end{eqnarray*}
Combining $e(G''_i)>1.05e(G_i)$ yields $\rho(G)>\rho^*$,
which contradicts the definition of $G^*$.
Therefore, we conclude that $G_i\notin\mathbb{G}'$
for any $i\in\{1,\ldots,s\}$.
\end{proof}

Given a rooted tree $T$,
a \emph{branching vertex} is a vertex with degree at least three.
An \emph{edge switching} on $T$ refers to the operation of constructing a new tree
by removing an edge $u_1v_1$ and adding a new edge $u_2v_2$.
A \emph{pruning} on $T$ is a specific case of edge switching. 
It involves constructing a new tree $T'=T-\{u_1v_1\}+\{v_1v_2\}$,
where $u_1$ is a branching vertex,
$v_1$ is its child, and $v_2$ is a leaf that is not a descendant of $v_1$.
We conclude this section by
proving Theorem \ref{thm1.3}.

\begin{proof}
Given Lemmas \ref{lem3.9} and \ref{lem3.10},
if the members in $\mathbb{G}$ are all small,
or if some member contains a bicyclic subgraph,
then $G^*-L\in EX(n-|L|,\Gamma(\mathbb{H})_{minor})$,
and the proof is complete.
Next, assume that
every member in $\mathbb{G}$ is either a tree or a unicyclic graph,
and that there exists a member $G_{i_0}\in\mathbb{G}$ such that
$|G_{i_0}|$ is sufficiently large.
Recall that $G^\star-L\in EX(n-|L|,\Gamma(\mathbb{H})_{minor})$.
Therefore, every component of $G^\star-L$ is a member of $\mathbb{G}$,
and hence, $e(G^\star-L)\leq |G^\star-L|=n-|L|$.

Let $G_i$ be an arbitrary member in $\mathbb{G}$, and let $V(G_i)=\cup_{k=1}^3U_k$,
where $U_k=\{v\in V(G_i): d_{G_i}(v)=k\}$ for $k\in\{1,2\}$,
and $U_3=V(G_i)\setminus(U_1\cup U_2)$.
Recall that every member in $\mathbb{G}$ is $K_{1,\alpha_\mathbb{H}}$-minor-free.
Thus, $\max\{\Delta(G_i),|U_1|\}<\alpha_\mathbb{H}.$
Since $G_i$ is either a tree or a unicyclic graph,
we have $e(G_i)\leq|G_i|=\sum_{i=1}^3|U_k|$.
Furthermore, 
$e(G_i)\geq\frac12(|U_1|+2|U_2|+3|U_3|).$
Combining the above three inequalities, we deduce that $|U_3|\leq|U_1|<\alpha_\mathbb{H}.$

On the one hand,
the inequality $\max\{\Delta(G_{i}),|U_3|\}<\alpha_\mathbb{H}$
implies that every $G_i$ can be transformed into a path or a lollipop graph
by at most ${\alpha_\mathbb{H}}^2$ steps of pruning.
On the other hand,
since $|U_1|+|U_3|<2\alpha_\mathbb{H}$ for $G_{i_0}$,
but $|G_{i_0}|$ is sufficiently large,
$G_{i_0}$ must contain a path of sufficient length.
Note that $G_{i_0}$ is $\Gamma(\mathbb{H})$-minor-free.
Thus, $P_k\notin \Gamma(\mathbb{H})$ for any fixed positive integer $k$.

Recall that $G^*-L$ contains $s$ non-isomorphic components $G_1,\ldots,G_s$,
each of which is either a tree or unicyclic.
Suppose that $G^*[L'']$ contains at least two tree components, denoted $G_i$ for $i\in\{1,2\}$.
Then, $G_i$ can be transformed into a path $P_{|G_i|}$
by at most ${\alpha_\mathbb{H}}^2$ steps of pruning.
This implies that $P_{|G_i|}=G_i-E_i'+E_i''$,
where $E_i'\subseteq E(G_i)$, $E_i''$ is set of non-edges of $G_i$,
and $|E_i'|=|E_i''|\leq {\alpha_\mathbb{H}}^2$.
Let $P$ be a path obtained from
$P_{|G_1|}\cup P_{|G_2|}$ by adding an edge.
Moreover, define $G$ as the graph obtained from $G^*$
by replacing $E(G_1\cup G_2)$ with $E(P)$.
Since $P_k\notin \Gamma(\mathbb{H})$ for any positive integer $k$,
$G-L$ remains $\Gamma(\mathbb{H})$-minor-free.
However, similarly to (\ref{align.21}), we have
\begin{eqnarray}\label{align.24}
\frac12\big(\rho(G)\!-\!\rho^*\big)\geq\!\!\!\!\!\sum\limits_{uv\in E(P)}\!\!\!\!x_ux_v-\!\!\!\!\!\!\sum\limits_{uv\in E(G_1\cup G_2)}
\!\!\!\!\!\!\!\!x_ux_v\geq \frac{(|E_1''\!\cup\!E_2''|\!+\!1)\sigma_L^2}{{\rho^*}^2}
-\frac{|E_1'\!\cup\!E_2'|\sigma_L^2}{(\rho^*\!-\!\alpha_{\mathbb{H}}\!+\!1)^2},
\end{eqnarray}
which implies that $\rho(G)>\rho^*$,
a contradiction.

Now, we conclude that $G^*-L$ contains at most one tree component.
If it does not contain any tree component,
then $e(G^*-L)=n-|L|\geq e(G^\star-L)$, and thus $G^*-L\in EX(n-|L|,\Gamma(\mathbb{H})_{minor})$,
as desired.
It remains to consider the case where $G^*-L$ has exactly one tree component, say $G_1$.
In this case, $e(G^*-L)=n-|L|-1$.
If $e(G^\star-L)\leq n-|L|-1$,
then $G^*-L\in EX(n-|L|,\Gamma(\mathbb{H})_{minor})$.
We may therefore assume that $e(G^\star-L)=n-|L|$,
which implies that every component of $G^\star-L$ is unicyclic.
We will now derive a contradiction to complete the proof.

We can select at most three components, $G_1,G_2,G_3$, of $G^*-L$,
such that $\sum_{i=1}^3|G_i|\geq6$, where $G_1$ is a tree, and $G_2$ and $G_3$ (if there exist) are unicyclic.
Recall that $G_1$ can be transformed to a path $P_{|G_1|}$
by at most ${\alpha_\mathbb{H}}^2$ steps of pruning,
and $G_2\cup G_3$ can be transformed to two lollipop graphs
by at most $2{\alpha_\mathbb{H}}^2$ steps of pruning.
Hence, $\cup_{i=1}^3G_i$ can be transformed into a cycle $C$ of length $\sum_{i=1}^3|G_i|$
by switching at most $3{\alpha_\mathbb{H}}^2+2$ edges and adding an edge.
Define $G=G^*-\cup_{i=1}^3E(G_i)+E(C)$.
Similarly to (\ref{align.24}), we obtain $\rho(G)>\rho^*$.
Let $|C|=3a+b$, where $b\in\{0,1,2\}$,
and let $G'$ be obtained from $G$
by replacing the cycle $C$ with $a-b$ copies of $C_3$ and $b$ copies of $C_4$.
Since both $G[V(C)]$ and $G'[V(C)]$ are $2$-regular, it is easy to verify that $\rho(G')=\rho(G)$.

Recall that every component of $G^\star-L$ is unicyclic.
Thus, $G^\star-L$ contains a $C_3$ minor,
which implies that $C_3$ is $\Gamma(\mathbb{H})$-minor-free.
If every unicyclic member of $\mathbb{G}$ is a triangle,
then $G^\star-L$ consists of $\frac{n-|L|}{3}$ copies of $C_3$,
so $3\mid (n-|L|)$.
Similarly, we have $3\mid (3a+b)$, and thus $b=0$. 
Now, $G'-L$ is $\Gamma(\mathbb{H})$-minor-free, but
$\rho(G')=\rho(G)>\rho(G^*)$, a contradiction.
Hence, $\mathbb{G}$ contains at least one unicyclic member that is not isomorphic to $C_3$.
This member must contain $H_0$ as a minor, where $H_0$ is either $C_4$ or $K_{1,3}+e$.
If $H_0\cong C_4$ or $b=0$, then $G'-L$ is still $\Gamma(\mathbb{H})$-minor-free, but
$\rho(G')=\rho(G)>\rho(G^*)$, a contradiction.
If $H_0\cong K_{1,3}+e$ and $b\neq0$, then we further define $G''$
as the graph obtained from $G'$
by replacing $b$ copies of $C_4$ with $b$ copies of $ K_{1,3}+e$.
By a well-known theorem of Wu, Xiao and Hong (see \cite{WU}, Theorem 1),
we can easily see that $\rho(G'')>\rho(G')>\rho(G^*)$.
But now, $G''-L$ is clearly $\Gamma(\mathbb{H})$-minor-free,
which is again a contradiction.
This completes the proof of Theorem \ref{thm1.3}.
\end{proof}

\section{Excluding complete $r$-partite minors}\label{section4}

In this section, we prove Theorem \ref{thm1.8}.
Recall that $\overline{G}$ denotes the complement of a graph $G$,
and $H_{s_1,s_2}=(\beta-1)K_{1,s_2}\cup K_{1,s_2+\beta_0}$,
where $0\leq \beta_0\leq s_2$ and $\beta(s_2+1)+\beta_0=s_1+1$.
Clearly, $H_{s_1,s_2}$ is a star forest of order $s_1+1$.
In particular, $H_{s_1,1}\cong\frac{s_1+1}{2}K_{1,1}$ for odd $s_1$,
and $H_{s_1,1}\cong\frac{s_1-2}{2}K_{1,1}\cup K_{2,1}$ for even $s_1$.
Let $S\big(\overline{H_{s_1,s_2}}\big)$ be the graph obtained
from $\overline{H_{s_1,s_2}}$ by subdividing
an edge whose endpoints have minimum degree sum.
One can observe that $\overline{H_{4,1}}\cong S^1(K_4)$ and
$S(\overline{H_{4,1}})\cong S^2(K_4)$.

Note that $SPEX(n,\{K_{s_1,s_2}\}_{minor})$ was determined for $s_1\geq4$ and $s_2\geq2$ in \cite{ZL},
and for the remaining cases in \cite{NIKI1,NIKI,WANG,ZHAI}.
These results resolved a conjecture posed by Tait \cite{Tait}.
In fact, these results can be rewritten as a slightly stronger version, as follows.

\begin{thm}\label{thm4.1}
Let $s_1\geq s_2\geq2$,
$\beta=\lfloor\frac{s_1+1}{s_2+1}\rfloor$ and $n-\gamma=ps_1+q$ \emph{($1\leq q\leq s_1$)}.
If $G$ is the join of a $\gamma$-clique
and an $(n-\gamma)$-vertex $\Gamma^*_{s_1+1}(K_{s_1,s_2})$-minor-free graph,
then $\rho(G)\leq \rho(K_\gamma\nabla G^\blacktriangle)$, with equality if and only if
$G\cong K_\gamma\nabla G^\blacktriangle$, where $G^\blacktriangle$ is defined as in (\ref{align.00}).
\end{thm}

We now characterize $SPEX(n,\{H\}_{minor})$
for $H=K_{s_1,\ldots,s_r}$, where $r\geq2$ and $s_1\geq\cdots\geq s_r$.
Clearly, $\alpha_H=s_1$ and
$\Gamma_{\alpha_H+1}^*(H)=\Gamma_{s_1+1}^*(H).$
We first consider the case when $s_2\geq2$.

\begin{thm}\label{thm4.2}
Let $\beta=\lfloor\frac{s_1+1}{s_2+1}\rfloor$,
$\gamma=\sum_{2}^{r}s_i-1$ and $n-\gamma=ps_1+q$ \emph{($1\leq q\leq s_1$)}.
If $s_2\geq2$, then $SPEX(n,\{K_{s_1,\ldots,s_r}\}_{minor})
=\{K_\gamma\nabla G^\blacktriangle\}$.
\end{thm}

\begin{proof}
For $H=K_{s_1,\ldots,s_r}$, let $\cup_{i=1}^rS_i$ be the $r$-partite partition of $V(H)$,
where $|S_i|=s_i$.
Given an arbitrary $(\alpha_H+1)$-subset $S$ of $V(H)$.
We first claim that there exists an $(\alpha_H+1)$-subset $S'$
such that $S'\subseteq S_1\cup S_2$
and $H[S']$ is isomorphic to a subgraph of $H[S]$.

Let $t_i=|S_i\cap S|$ for $i\in\{1,\ldots,r\}$.
Up to isomorphism of $H[S]$, one can assume that $t_1\geq\cdots\geq t_r$.
If $t_3=0$, then we choose $S'=S$, as required.
Suppose now that $t_3\geq1$. Then $t_2\geq1$.
Note that $\sum_{i=1}^rt_i=|S|=s_1+1$.
Thus, $\sum_{i=3}^rt_i=s_1+1-t_1-t_2\leq s_1-t_1=|S_1\setminus S|.$
Now, let $S'$ be obtained from $S$ by replacing $\sum_{i=3}^rt_i$ vertices
in $S\setminus(S_1\cup S_2)$ with $\sum_{i=3}^rt_i$ vertices in
$S_1\setminus S$.
Then, $S'\subseteq S_1\cup S_2$,
and it is easy to see that
$H[S']$ is isomorphic to a subgraph of $H[S]$.
Hence, the claim holds, which
further implies
$\Gamma_{\alpha_H+1}^*(H)=\Gamma_{s_1+1}^*(K_{s_1,s_2})$.

By Theorem \ref{thm1.1}, every graph $G^*$ in $SPEX(n,\{H\}_{minor})$ contains a set $L$ of
dominating vertices, where $|L|=|H|-\alpha_H-1=\sum_{2}^{r}s_i-1$.
By Theorem \ref{thm1.2}, $G^*-L$ is $\Gamma_{s_1+1}^*(K_{s_1,s_2})$-minor-free.
Now, setting $\gamma=\sum_{2}^{r}s_i-1$ in Theorem \ref{thm4.1},
we immediately obtain that $G^*\cong K_\gamma\nabla G^\blacktriangle$.
\end{proof}

Given Theorem \ref{thm4.2}, it remains to characterize $SPEX(n,\{H\}_{minor})$
for the complete $r$-partite graph $H=K_{s_1,1,\ldots,1}$.
Now, $H$ is a book $B_{r-1,s_1}$,
and for $r=2$, $SPEX(n,\{K_{s_1,1}\}_{minor})$ was determined in \cite{ZL}.
Hence, we assume that $\sum_{2}^{r}s_i\geq 2$ in Theorem \ref{thm1.8},
which implies that $r\geq3$.
In this case, we have
$\alpha_H=s_1$, $\gamma_H=|H|-\alpha_H-1=r-2\geq1$, and
\begin{eqnarray}\label{align.16}
\Gamma^*_{\alpha_H+1}(H)=\Gamma^*_{s_1+1}(K_{s_1,1,\ldots,1})=\{K_{s_1,1}\}.
\end{eqnarray}

Assume that $G^*\!\in\!SPEX(n,\{B_{r-1,s_1}\}_{minor})$,
and that $X^*\!=\!(x_1,\ldots,x_n)^T$ is the positive
unit eigenvector corresponding to $\rho^*:=\rho(G^*)$.
By Theorem \ref{thm1.1}, we have $|L|=\gamma_H=r-2$.

\begin{lem}\label{lem4.1}
Let $\gamma:=r-2$. Then, we have the inequality:
\begin{eqnarray}\label{align.17}
{\rho^*}^2-(s_1+\gamma-2)\rho^*\leq\gamma(n-\gamma)-(\gamma-1)(s_1-1),
\end{eqnarray}
with equality if and only if $G^*-L$ is
an $(s_1-1)$-regular $K_{s_1,1}$-minor-free graph.
\end{lem}

\begin{proof}
By Theorem \ref{thm1.2}, and in view of (\ref{align.16}),
we can observe that
$G^*-L$ is $K_{s_1,1}$-minor-saturated. 
Hence, $\Delta(G^*-L)\leq s_1-1$.
By symmetry, $x_u$ is constant for $u\in L$.
Choose $u^*\in L$
and $v^*\in V(G^*)\setminus L$ such that $x_{v^*}=\max_{v\in V(G^*)\setminus L}x_v.$
Note that $\rho^*x_{u}=\sum_{v\in N_{G^*}(u)}x_v$ for $u\in V(G^*)$. Thus,
\begin{eqnarray}\label{align.18}
\rho^*x_{v^*}\leq\gamma x_{u^*}+(s_1-1)x_{v^*}~~
\mbox{and}~~
\rho^*x_{u^*}\leq(\gamma-1)x_{u^*}+(n-\gamma)x_{v^*}.
\end{eqnarray}
Combining the two inequalities in (\ref{align.18}),
we obtain (\ref{align.17}) immediately.

Next, we characterize the equality case in (\ref{align.17}).
If the equality holds, then both
inequalities in (\ref{align.18}) must become equalities.
This implies that $d_{G^*-L}(v)=s_1-1$ for each $v\in V(G^*)\setminus L$,
meaning that $G^*-L$ is $(s_1-1)$-regular.
Conversely, if $G^*-L$ is $(s_1-1)$-regular,
then $G^*$ is the join of two regular graphs, which implies that
both $X^*\parallel_{L}$ and $X^*\parallel_{V(G^*)\setminus L}$
are constant vectors.
As a result, both inequalities in (\ref{align.18}) hold as equalities,
and consequently, (\ref{align.17}) holds as an equality.
\end{proof}

For an arbitrary $v\in V(G^*)\setminus L$,
we have $\rho^*x_v=\gamma x_{u^*}+\sum_{w\in N_{G^*}(v)\setminus L}x_w$,
and thus, $\rho^*x_v\geq\gamma x_{u^*}.$
Based on (\ref{align.18}), we obtain $(\rho^*-s_1+1)x_v\leq(\rho^*-s_1+1)x_{v^*}\leq\gamma x_{u^*}.$
Hence,
$\frac{\gamma}{\rho^*}x_{u^*}\leq x_v\leq\frac{\gamma}{\rho^*-s_1+1}x_{u^*}.$
For any $u,v\in V(G^*)\setminus L$ with $d_{G^*}(u)<d_{G^*}(v),$ we can further deduce that
\begin{eqnarray}\label{align.19}
x_u<x_v.
\end{eqnarray}

\begin{lem}\label{lem4.2}
Assume that $s_1\neq3$ and that $G$ is a connected
$K_{s_1,1}$-minor-free graph.
If $G$ is $(s_1-1)$-regular, then either $G\cong K_{s_1}$, or
$G\cong\overline{H_{s_1,1}}$ only for odd $s_1$.
\end{lem}

\begin{proof}
We shall first note that both $K_{s_1}$ and $\overline{H_{s_1,1}}$
are $K_{s_1,1}$-minor-free.
Indeed, recall that $H_{s_1,1}\cong\frac{s_1+1}{2}K_{1,1}$ for odd $s_1$
and $H_{s_1,1}\cong\frac{s_1-2}{2}K_{1,1}\cup K_{2,1}$ for even $s_1$.
Thus, $|H_{s_1,1}|=s_1+1$ and $\Delta(\overline{H_{s_1,1}})=s_1-1$.
Therefore, $\overline{H_{s_1,1}}$
is $K_{s_1,1}$-minor-free,
and so is $K_{s_1}$.

Note that $G$ is connected and $(s_1-1)$-regular.
If $s_1\in\{1,2\}$ or $|G|=s_1$, then $G\cong K_{s_1}$.
If $|G|=s_1+1$, then $G$ can only be obtained from $K_{s_1+1}$ by
deleting a perfect matching,
which implies that $s_1$ is odd and $\overline{G}\cong\frac{s_1+1}2K_{1,1}\cong H_{s_1,1}.$
Consequently, $G\cong\overline{H_{s_1,1}}$.

Next, assume that $s_1\geq4$ and $|G|\geq s_1+2$.
Since $G$ is connected and $K_{s_1,1}$-minor-free,
by Lemma \ref{lem3.1} we have
$e(G)\leq{s_1\choose 2}+|G|-s_1$,
which implies that $e(G)<\frac12(s_1-1)|G|$.
Thus, $G$ cannot be $(s_1-1)$-regular,
a contradiction.
Hence, the lemma holds.
\end{proof}

From Lemmas \ref{lem4.1} and \ref{lem4.2},
we immediately obtain the following theorem.

\begin{thm}\label{thm4.3}
If $s_1$ is odd,
then $SPEX(n,\{B_{r-1,s_1}\}_{minor})\!=\!\{K_\gamma\nabla G^\blacktriangledown\}$,
where $\gamma\!=\!r-2$ and $G^\blacktriangledown$ takes over all 
$(n\!-\!\gamma)$-vertex $(s_1\!-\!1)$-regular $K_{s_1,1}$-minor-free graphs.
More precisely, every component of $G^\blacktriangledown$ is a cycle when $s_1=3$,
and is either $K_{s_1}$ or $\overline{H_{s_1,1}}$ otherwise.
\end{thm}

\begin{remark}
By Theorem \ref{thm4.3},
$SPEX(n,\{B_{r-1,s_1}\}_{minor})$
is an infinite family for odd $s_1\geq3$.
Indeed, let $s_1\geq 5$ and $n-\gamma=ps_1+q$ \emph{($1\leq q\leq s_1$)}.
Then, $G^\blacktriangledown$ can be constructed as
the union of $p-c-(cs_1+q)$ copies of $K_{s_1}$
and $cs_1+q$ copies of $\overline{H_{s_1,1}}$,
where $c$ is an arbitrary non-negative integer.
\end{remark}

Let $\mathbb{G}_i$ denote the family of $i$-vertex components in $G^*-L$,
and let $|\mathbb{G}_i|$ represent the number of components in $\mathbb{G}_i$.
In the following, we characterize $G^*$ for even $s_1$.

\begin{lem}\label{lem4.3}
Let $s_1$ be even and $\mathbb{G}_i\neq\varnothing$. 
Then, for $s_1=4$, we have $s_1-1\leq i\leq s_1+3$, and for $s_1\geq4$, we have $s_1-1\leq i\leq s_1+1$.
\end{lem}

\begin{proof}
Recall that $G^*-L$ is $K_{s_1,1}$-minor-saturated. 
Assume that $s_1\geq4$. We first claim that $\mathbb{G}_i=\varnothing$ for each $i\geq3s_1$.
Indeed, if $G^*-L$ contains a component
$G_0$ with $|G_0|=as_1+b$, where $a\geq3$ and $0\leq b<s_1,$
then by Lemma \ref{lem3.1},
$e(G_0)\leq{s_1\choose2}+|G_0|-s_1<{s_1\choose2}+as_1$.
Now, let $G_0'=aK_{s_1}\cup bK_1$.
Then, $G_0'$ is $K_{s_1,1}$-minor-free and
$e(G_0')=a{s_1\choose2}$.
A straightforward calculation gives $e(G_0)<e(G_0')$
for $a\geq3$ and $s_1\geq4$,
which contradicts Theorem \ref{thm1.3}.

Secondly, we claim that $|\mathbb{G}_i|\leq s_1-1$ for
$i\neq s_1$.
Otherwise, there exists some
$i_1\neq s_1$ such that $|\mathbb{G}_{i_1}|\geq s_1$.
Let $G_1$ be the union of $s_1$ components in $\mathbb{G}_{i_1}$.
Then $\Delta(G_1)\leq s_1-1$, as $G_1$ is $K_{s_1,1}$-minor-free.
Moreover, $G_1$ is not $(s_1-1)$-regular
(otherwise, $G_1\cong s_1K_{s_1}$ by Lemma \ref{lem4.2}).
Thus,
$e(G_1)<\frac12(s_1-1)|G_1|=e(i_1K_{s_1}),$
which also contradicts Theorem \ref{thm1.3}.

The above two claims imply that
$\sum_{i\neq s_1}i|\mathbb{G}_i|$ is constant.
Hence, $s_1|\mathbb{G}_{s_1}|=n-\gamma-\sum_{i\neq s_1}i|\mathbb{G}_i|$
$\geq\frac{n}{2s_1}$.
For an arbitrary $i_2\in\{1,2,\ldots,3s_1-1\}\setminus\{s_1,s_1+1\}$,
we set $i_2=as_1+b$, where $0\leq a\leq 2$ and $0\leq b<s_1$.
If $\mathbb{G}_{i_2}\neq\varnothing$,
we choose a subgraph $G_2$ in $G^*-L$,
which consists of $b$ components from $\mathbb{G}_{s_1}$
and one component from $\mathbb{G}_{i_2}$.
For $a=0$, it is clear that $e(G_2)\leq b\cdot e(K_{s_1})+e(K_{b})$.
For $a\in\{1,2\}$, by Lemma \ref{lem3.1}, we get
$e(G_2)\leq b\cdot e(K_{s_1})+{s_1\choose 2}+(i_2-s_1)$.
Next, let $G_2'=aK_{s_1}\cup b\overline{H_{s_1,1}}.$
Then, $|G_2'|=|G_2|$ and $e(G_2')=a\cdot
e(K_{s_1})+b\cdot\frac12(s_1^2-2).$
A straightforward calculation shows that
$e(G_2)\le e(G_2'),$
with equality if and only if $i_2=s_1-1$ or $i_2\in\{s_1+2,s_1+3: ~s_1=4\}$.

Recall that
both $K_{s_1}$ and $\overline{H_{s_1,1}}$ are $K_{s_1,1}$-minor-free.
Therefore, $G_2'$ is also $K_{s_1,1}$-minor-free.
By Theorem \ref{thm1.3}, we have
$e(G_2)\geq e(G_2')$.
Hence, $e(G_2)=e(G_2')$.
Consequently,
$i_2=s_1-1$ or $i_2\in\{s_1+2,s_1+3: ~s_1=4\}$.
Given the choice of $i_2$, we have now completed the proof.
\end{proof}

\begin{lem}\label{lem4.4}
Let $s_1\geq4$ be even.
Then, $|\mathbb{G}_{s_1-1}|\leq 1$ and $|\mathbb{G}_{s_1+1}|\leq s_1-2$.
Specially, if $|\mathbb{G}_{s_1-1}|=1$, then
$\mathbb{G}_i=\varnothing$ for any $i\notin\{s_1-1,s\}$.
\end{lem}

\begin{proof}
We first assume that there exists a component $G_0\in \mathbb{G}_{s_1-1}$.
Since $G^*-L$ is $K_{s_1,1}$-minor-saturated,
we have $G_0\cong K_{s_1-1}$.
Choose an arbitrary component $G_1$ in $G^*-L$ other than $G_0$.
We now claim that $|G_1|=s_1$.
Indeed, suppose for contradiction that $|G_1|\neq s_1$.
Then, by Lemma \ref{lem4.2}, $G_1$ is not $(s_1-1)$-regular,
and thus there exists $v\in V(G_1)$ with $d_{G_1}(v)\leq s_1-2$.
Let $G$ be the graph obtained from $G^*$ by replacing $G_0\cup G_1$ with
$K_{s_1}\cup (G_1-\{v\})$.
Note that $e(K_{s_1})-e(G_0)=s_1-1$, but $e(G_1)-e(G_1-\{v\})\leq s_1-2.$
Hence, $e(G)>e(G^*)$, which contradicts Theorem \ref{thm1.3}.
Therefore, $|G_1|=s_1$, as claimed.
Now, we conclude that if there exists a component $G_0\in \mathbb{G}_{s_1-1}$,
then every other component in $G^*-L$ must belong to $\mathbb{G}_{s_1}$.
This implies that
$|\mathbb{G}_{s_1-1}|\leq 1$.
Furthermore, if $|\mathbb{G}_{s_1-1}|=1$, then
$\mathbb{G}_i=\varnothing$ for any $i\notin\{s_1-1,s\}$.

It remains to show that $|\mathbb{G}_{s_1+1}|\leq s_1-2$.
Suppose to the contrary that 
$|\mathbb{G}_{s_1+1}|\geq s_1-1$, and
let $G_2$ be a component in $\mathbb{G}_{s_1+1}.$
By Theorem \ref{thm1.3}, $e(G_2)=ex(s_1+1,\{K_{s_1,1}\}_{minor})$.
Thus, $G_2$ can only be the complement of $\frac{s_1-2}2K_{1,1}\cup K_{1,2}$,
i.e., $G_2\cong\overline{H_{s_1,1}}$.
Now, let $G$ be the graph obtained from $G^*$ by replacing
$(s_1-1)G_2$ with
$(s_1-1)K_{s_1}\cup K_{s_1-1}$.
By Lemma \ref{lem3.6},
$G$ is also $B_{r-1,s_1}$-minor-free.

Assume now that $V(G_2)=
\{v_0,v_1,\ldots,v_{s_1}\}$,
where the set of non-edges in $G_2$ is formed by $\{v_0v_1,v_0v_2\}\cup\{v_3v_4,v_5v_6,\ldots,v_{s_1-1}v_{s_1}\}$.
By symmetry, we have $x_{v_i}=x_{v_3}$ for each $i\in\{3,\ldots,s_1\}$.
Furthermore, since $d_{G_2}(v_0)=s_1-2$ and $d_{G_2}(v_3)=s_1-1$,
it follows that $x_{v_0}<x_{v_3}$ by (\ref{align.19}).

Observe that $(s_1-1)K_{s_1}\cup K_{s_1-1}$
can be obtained from $(s_1-1)G_2$
by replacing the edge set $\{v_0v_i:~ i=3,\ldots,s_1\}$ with $\{v_3v_4,v_5v_6,\ldots,v_{s_1-1}v_{s_1}\}$ in every copy of $G_2$
and then forming $s_1-1$ copies of $v_0$ into a copy of $K_{s_1-1}$.
Thus, $e(G)=e(G^*)$, and we can deduce that
\begin{eqnarray*}
\sum_{uv\in E(G)}\!\!\!\!2x_ux_v-\!\!\!\!\sum_{uv\in E(G^*)}\!\!\!\!2x_ux_v
\!\!\!\!\!&=&\!\!\!\!\! 2e(K_{s_1-1})x_{v_0}^2+2(s-1)\big(\sum_{i=2}^{s_1/2}x_{v_{2i-1}}x_{v_{2i}}-
\sum_{i=3}^{s_1}x_{v_0}x_{v_i}\big)\\
\!\!\!\!\!&=&\!\!\!\!\! (s_1-1)(s_1-2)(x^2_{v_0}+x^2_{v_3}-2 x_{v_0}x_{v_3}).
\end{eqnarray*}
Since $x_{v_0}<x_{v_3},$
we have $\rho(G)>\rho^*$, a contradiction.
Therefore, $|\mathbb{G}_{s_1+1}|\leq s_1-2$.
\end{proof}

\begin{lem}\label{lem4.5}
If $s_1=4$,
then we have $\mathbb{G}_{s_1+3}=\varnothing$,
and $\sum_{i\in\{-1,1,2\}}|\mathbb{G}_{s_1+i}|\leq1$.
\end{lem}

\begin{proof}
Let $G_0$ be an arbitrary component in $G^*-L$.
Then, by Lemma \ref{lem4.3}, we have $|G_0|\leq s_1+3$,
and by Theorem \ref{thm1.3},
$e(G_0)=ex(|G_0|,\{K_{s_1,1}\}_{minor})$.
Furthermore,
by Lemma \ref{lem3.2},
every member in $\mathbb{G}_{s_1+i}$
is isomorphic to $S^i(K_{s_1})$ for $i\in\{1,2,3\}$.

We first show that $\mathbb{G}_{s_1+3}=\varnothing.$
Suppose, for contradiction, that there exists a component $G_0\in\mathbb{G}_{s_1+3}$.
Then, $G_0\cong S^3(K_4)$,
meaning that $G_0$ is obtained from $K_4$
by replacing an edge $v_1v_2$ with a path $v_1w_1w_2w_3v_2$.
Now, let $G_1=G_0-\{v_1w_1,v_2w_3\}+\{v_1v_2,w_1w_3\}.$
Then, $G_1\cong K_4\cup K_3$, and clearly $G_1$ is $K_{1,4}$-minor-free.
Define $G$ to be the graph obtained from $G^*$ by replacing $G_0$ with $G_1$.
By Lemma \ref{lem3.6}, $G$ is $B_{r-1,s_1}$-minor-free. Moreover, we have 
$$\rho(G)-\rho^*\geq
2(x_{v_1}x_{v_2}+x_{w_1}x_{w_3})-2(x_{v_1}x_{w_1}+x_{v_2}x_{w_3}).$$
By symmetry, $x_{v_1}=x_{v_2}$ and $x_{w_1}=x_{w_3}$.
Consequently, $\rho(G)-\rho^*\geq2(x^2_{v_1}+x^2_{w_1}-2x_{v_1}x_{w_1})$.
Note that $d_{G_0}(w_1)=2$ and $d_{G_0}(v_1)=3$.
By (\ref{align.19}),
we obtain
$x_{w_1}<x_{v_1},$
and hence $\rho(G)>\rho^*$, which leads to a contradiction.
Thus, $\mathbb{G}_{s_1+3}=\varnothing,$ as desired.

Secondly,
we claim that
$|\mathbb{G}_{s_1+i}|\leq1$ for $i\in\{1,2\}$.
Indeed, if $|\mathbb{G}_{s_1+2}|\geq2$, we replace
two copies of $S^2(K_4)$ in $\mathbb{G}_{s_1+2}$ with three copies of $K_4$.
Since $2e(S^2(K_4))=16<3e(K_4)$,
this results in a contradiction by Theorem \ref{thm1.3}.
Now, if $|\mathbb{G}_{s_1+1}|\geq2$,
we choose $G_0,G_1\in\mathbb{G}_{s_1+1}$.
For $j\in\{0,1\}$, $G_j\cong S^1(K_4)$, meaning $G_j$ is obtained from $K_4$
by replacing an edge $u_jw_j$ with a path $u_jv_jw_j$.
Now, define $G_2=(G_0\cup G_1)-\{u_1v_1,v_1w_1,u_0v_0\}+\{u_1w_1,u_0v_1,v_1v_0\}.$
Clearly, $G_2\cong K_4\cup S^2(K_4)$, and hence, $G_2$ is $K_{1,4}$-minor-free.
Let $G$ be the graph obtained from $G^*$ by replacing $G_0\cup G_1$ with $G_2$.
By Lemma \ref{lem3.6}, $G$ is $B_{r-1,s_1}$-minor-free, and we have
$$\rho(G)-\rho^*\geq
2(x_{u_1}x_{w_1}+x_{u_0}x_{v_1}+x_{v_1}x_{v_0})
-2(x_{u_1}x_{v_1}+x_{v_1}x_{w_1}+x_{u_0}x_{v_0}).$$
By symmetry, we know that $x_{v_0}=x_{v_1}$ and $x_{u_0}=x_{u_1}=x_{w_0}=x_{w_1}$.
Thus, $\rho(G)-\rho^*\geq2(x^2_{u_0}+x^2_{v_0}-2x_{u_0}x_{v_0})$.
Since $d_{G_0}(v_0)=2$ and $d_{G_0}(u_0)=3$,
by inequality (\ref{align.19}), we have $x_{v_0}<x_{u_0}$,
and hence $\rho(G)>\rho^*$, a contradiction.
Therefore, the claim holds.

Now, we are ready to prove that
$\sum_{i\in\{-1,1,2\}}|\mathbb{G}_{s_1+i}|\leq1$.
If $\mathbb{G}_{s_1-1}\neq\varnothing$, then by Lemma \ref{lem4.4}, we are done. 
Next, assume that $\mathbb{G}_{s_1-1}=\varnothing$.
It suffices to show that
$\sum_{i=1}^2|\mathbb{G}_{s_1+i}|\leq1$.
Suppose, for contradiction, that this is not the case.
By the above claim,
there must simultaneously exist $G_1\in \mathbb{G}_{s_1+1}$ and
$G_2\in \mathbb{G}_{s_1+2}$, where
$G_i\cong S^i(K_4)$ for $i\in\{1,2\}$.
Let $$G_3=(G_1\cup G_2)-\{u_1u_2,u_2u_3,v_1v_2,v_3v_4\}
+\{u_1u_3,v_1v_4,u_2v_2,u_2v_3\},$$
where $u_1u_2u_3$ is the induced path of length two in $G_1$
and $v_1v_2v_3v_4$ is the induced path of length three in $G_2$.
Then, $G_3\cong 2K_4\cup K_3$, and clearly, $G_3$ is $K_{1,4}$-minor-free.
Define $G$ to be the graph obtained from $G^*$ by replacing $G_1\cup G_2$ with $G_3$.
Then, $G$ is $B_{r-1,s_1}$-minor-free.
By symmetry, we have $x_{u_1}=x_{u_3}$, $x_{v_1}=x_{v_4}$ and $x_{v_2}=x_{v_3}$.
Thus, we can deduce that
\begin{eqnarray*}
\rho(G)-\rho^*
\!\!\!\!&\geq&\!\!\!\!
2(x_{u_1}x_{u_3}+x_{v_1}x_{v_4}+x_{u_2}x_{v_2}+x_{u_2}x_{v_3})
-2(x_{u_1}x_{u_2}+x_{u_2}x_{u_3}+x_{v_1}x_{v_2}+x_{v_3}x_{v_4})\\
\!\!\!\!&=&\!\!\!\! 2(x^2_{u_1}+x^2_{v_1}+2x_{u_2}x_{v_2})
-2(2x_{u_1}x_{u_2}+2x_{v_1}x_{v_2})\\
\!\!\!\!&=&\!\!\!\! 2(x_{u_1}-x_{v_1})^2+4(x_{u_1}-x_{v_2})(x_{v_1}-x_{u_2}).
\end{eqnarray*}
Note that $d_{G_1}(u_2)=d_{G_2}(v_2)=2$
and $d_{G_1}(u_1)=d_{G_2}(v_1)=3$.
By (\ref{align.19}), we obtain
$\max\{x_{u_2},x_{v_2}\}<\min\{x_{u_1},x_{v_1}\}$.
Hence, $\rho(G)>\rho^*$, a contradiction.
Therefore, $\sum_{i\in\{-1,1,2\}}|\mathbb{G}_{s_1+i}|\leq1$.
\end{proof}

We now determine $SPEX\big(n,\{B_{r-1,s_1}\}_{minor}\big)$ for even $s_1$.

\begin{thm}\label{thm4.4}
Let $n-\gamma=ps_1+q$, where $\gamma=r-2$ and $1\leq q\leq s_1$.
If $s_1$ is even, then we have $SPEX(n,\{B_{r-1,s_1}\}_{minor})
=\{K_\gamma\nabla G^\vartriangle\}$, where
$$G^\vartriangle=\left\{
\begin{aligned}
   &(p-1)K_{s_1} \cup S(\overline{H_{s_1,1}}) && ~\hbox{if}~ (q,s_1)=(2,4);\\
   &(p-q)K_{s_1} \cup q\overline{H_{s_1,1}} && ~\hbox{if}~ q\leq s_1-2 ~\hbox{and} ~ (q,s_1)\neq(2,4);\\
   &pK_{s_1}\cup K_q && \hbox{if}~ q>s_1-2.
\end{aligned}
\right.
$$
\end{thm}

\begin{proof}
Let $G^*\in SPEX(n,\{B_{r-1,s_1}\}_{minor})$.
By Theorems \ref{thm1.1} and \ref{thm1.2},
$G^*$ has a set $L$ of dominating vertices,
where $|L|=\gamma$, and $G^*-L$ is $K_{s_1,1}$-minor-saturated.
Let $G_0$ be a component in $G^*-L$.
By Theorem \ref{thm1.3}, $e(G_0)=ex(|G_0|,\{K_{s_1,1}\}_{minor})$.
Therefore, $G_0\cong K_{|G_0|}$ if $|G_0|\in\{s_1-1,s_1\}$, and
$G_0\cong \overline{H_{s_1,1}}$ if $|G_0|=s_1+1$.
By Lemma \ref{lem3.2}, if $|G_0|=s_1+2=6$, then
$G_0\cong S^2(K_4)\cong S(\overline{H_{s_1,1}})$.
In the following, we distinguish
the proof into three cases.

If $s_1=2$, then $q\in\{1,2\}=\{s_1-1,s_1\}$,
and $G^*-L$ is $K_{1,2}$-minor saturated.
It is easy to see that $G^*-L\cong pK_{s_1}\cup K_q$,
as desired.

If $s_1=4$, then $1\le q\leq4$, and by Lemmas \ref{lem4.3} and \ref{lem4.5},
we get $|G_0|\in \{s_1+i: -1\leq i\leq2\}$.
From Lemma \ref{lem4.5}, we also know that $\sum_{i\in\{-1,1,2\}}|\mathbb{G}_{s_1+i}|\leq1$.
Thus, if $q=1=s_1-3$, then $|G_0|\in\{s_1,s_1+1\}$, and
$G_0\in\{K_{s_1},\overline{H_{s_1,1}}\}$ as stated above,
which gives $G^*-L\cong(p-1)K_{s_1}\cup\overline{H_{s_1,1}}$.
If $q=2=s_1-2$, then $|G_0|\in\{s_1,s_1+2\}$, and
$G_0\in\{K_{s_1},S(\overline{H_{s_1,1}})\}$ as discussed above,
which implies that $G^*-L\cong (p-1)K_{s_1}\cup S(\overline{H_{s_1,1}})$.
If $q\in\{3,4\}=\{s_1-1,s_1\}$, then $|G_0|\in\{s_1-1,s_1\}$, and thus
$G_0\cong K_{|G_0|}$.
It follows that $G^*-L\cong pK_{s_1}\cup K_q$.

If $s_1\geq6$, then by Lemma \ref{lem4.3},
we know that $|G_0|\in\{s_1+i: -1\leq i\leq1\}$.
From Lemma \ref{lem4.4}, we further have $|\mathbb{G}_{s_1-1}|\leq 1$, $|\mathbb{G}_{s_1+1}|\leq s_1-2$,
and $\mathbb{G}_{s_1+1}=\varnothing$ provided that $|\mathbb{G}_{s_1-1}|=1$.
Thus, if $q\in\{s_1-1,s_1\}$,
then $|G_0|\in\{s_1-1,s_1\}$ and so $G_0\cong K_{|G_0|}$,
which implies that $G^*-L\cong pK_{s_1}\cup K_{q}$.
If $q\leq s_1-2$,
then $|G_0|\in\{s_1,s_1+1\}$, and thus
$G_0\in\{K_{s_1},\overline{H_{s_1,1}}\}$,
which implies that $G^*-L\cong (p-q)K_{s_1}\cup q\overline{H_{s_1,1}}$.
This completes the proof.
\end{proof}

We conclude this section with the proof of Theorem \ref{thm1.8}.

\begin{proof}
Let $H=K_{s_1,s_2,\ldots,s_r}$, where $s_1\geq s_2\geq\cdots\geq s_r\geq1$,
and $\gamma=\sum_{i=2}^rs_i-1\geq1$.
Let $\beta=\lfloor\frac{s_1+1}{s_2+1}\rfloor$.
If $s_2\geq2$,
then Theorem \ref{thm1.8} holds by Theorem \ref{thm4.2}.
If $s_2=1$ and $s_1$ is odd,
then Theorem \ref{thm1.8} holds by Theorem \ref{thm4.3}.
If $s_2=1$ and $s_1$ is even,
then $\beta=\lfloor\frac{s_1+1}{s_2+1}\rfloor=\frac{s_1}2$.
Now, $2(\beta-1)=s_1-2$, and the case $(q,\beta,s_1)=(2,1,8)$ never occurs.
Hence, $G^\vartriangle=G^\blacktriangle$,
and Theorem \ref{thm1.8} follows from Theorem \ref{thm4.4}.
This completes the proof of Theorem \ref{thm1.8}.
\end{proof}

\section{Proof of the stability theorem}\label{section5}

In this section, we prove Theorem \ref{thm2.1}.
Let $G$ be an $n$-vertex $\mathbb{H}$-minor free graph
with $\rho(G)\geq\sqrt{\gamma_{\mathbb{H}}(n-\gamma_{\mathbb{H}})}$,
and let $X$ be a non-negative eigenvector corresponding to $\rho(G)$.
Choose $u^*\in V(G)$ such that $x_{u^*}=\max_{u\in V(G)}x_u$,
To prove Theorem \ref{thm2.1},
we need to find a set $L$ of $\gamma_\mathbb{H}$ vertices in $G$
such that $x_u\ge\big(1-\frac{1}{2(10C_\mathbb{H})^2}\big)x_{u^*}$ and $d_G(u)\ge\big(1-\frac{1}{(10C_\mathbb{H})^2}\big)n$
for every $u\in L$.
To this end, we choose a subset of $V(G)$ with respect to a positive constant $\lambda$ as follows:
$$L^{\lambda}=\{u\in V(G):                                    ~x_u\ge (10C_\mathbb{H})^{-\lambda}x_{u^*}\}.$$

\begin{lem}\label{lem2.5}
$|L^{\lambda}|<(10C_\mathbb{H})^{\lambda-10}n$.
\end{lem}

\begin{proof}
Given an arbitrary $u\in L^\lambda$,
we have
$\rho(G)x_u\geq
\sqrt{\gamma_{\mathbb{H}}(n-\gamma_{\mathbb{H}})}(10C_\mathbb{H})^{-\lambda}x_{u^*}$,
and so
$\rho(G)x_u>2C_\mathbb{H}(10C_\mathbb{H})^{10-\lambda}x_{u^*}$
for sufficiently large $n$.
On the other hand,
$\rho(G)x_u=\sum_{v\in N_G(u)}x_v$
$\le d_G(u)x_{u^*}$.
Combining the above two inequalities gives 
$2C_\mathbb{H}(10C_\mathbb{H})^{10-\lambda}<d_G(u)$.
Summing this inequality
over all vertices $u\in L^{\lambda}$,
we obtain
\begin{align*}
2C_\mathbb{H}(10C_\mathbb{H})^{10-\lambda}|L^{\lambda}|
\!<\!\sum_{u\in L^{\lambda}}\!_G(u)
\!\le\!\sum_{u\in V(G)}\!d_G(u)=2e(G).
\end{align*}
Recall that $e(G)<C_\mathbb{H}n$ by Lemma \ref{lem2.2}. Thus,
$|L^{\lambda}|<(10C_\mathbb{H})^{\lambda-10}n$.
\end{proof}

For a vertex $u\in V(G)$,
let $N_i(u)$ denote the set of vertices at distance $i$ from $u$ in $G$.
Furthermore, we use $L_i^{\lambda}\!(u)$ and $\overline{L}_i^{\lambda}\!(u)$
to denote $N_i(u)\cap L^{\lambda}$ and $N_i(u)\setminus L^{\lambda}$, respectively.
We also denote $L_{i,j}^{\lambda}(u)=L_i^{\lambda}\!(u)\cup L_j^{\lambda}\!(u)$
and $\overline{L}_{i,j}^{\lambda}(u)=\overline{L}_i^{\lambda}\!(u)\cup\overline{L}_j^{\lambda}\!(u)$ for simplicity.

\begin{lem}\label{lem2.6}
For every $u\in V(G)$
and every positive constant $\lambda$, we have
\begin{eqnarray}\label{align.02}
\gamma_{\mathbb{H}}(n-\gamma_{\mathbb{H}})x_u\leq |N_1(u)|x_u\!+\!\Big(\frac{2C_\mathbb{H}n}{(10C_\mathbb{H})^{10-\lambda}}
\!+\!\frac{2C_\mathbb{H}n}{(10C_\mathbb{H})^{\lambda}}\Big)x_{u^*}
\!+\!\!\!\!\sum_{v\in\overline{L}_1^{\lambda}\!(u)\!, w\in N_{L_{1,2}^{\lambda}\!(u)}\!(v)}\!\!\!\!x_w.
\end{eqnarray}
\end{lem}

\begin{proof}
Set $\rho:=\rho(G)$. Recall that $\rho\geq\sqrt{\gamma_{\mathbb{H}}(n-\gamma_{\mathbb{H}})}$. 
Then, we have
\begin{eqnarray}\label{align.03}
\gamma_{\mathbb{H}}(n-\gamma_{\mathbb{H}})x_{u}\le{\rho}^2x_{u}
=\!\!\sum_{v\in N_1(u)}\rho x_v
=|N_1(u)|x_u+\!\!\!\!\sum_{v\in N_1(u), w\in N_1(v)\setminus\{u\}}\!\!\!\!x_w.
\end{eqnarray}

Observe that $N_1(v)\setminus\{u\}\subseteq\cup_{i=1}^2N_i(u)=L_{1,2}^{\lambda}(u)\cup \overline{L}_{1,2}^{\lambda}(u)$
for any $u\in V(G)$ and any $v\in N_1(u)$. 
To bound the sum on the right-hand side of (\ref{align.03}),
we partition $N_1(u)$ into $L_1^{\lambda}\!(u)\cup\overline{L}_1^{\lambda}\!(u)$
and $N_1(v)\setminus\{u\}$ into $N_{L_{1,2}^{\lambda}(u)}\!(v)\cup N_{\overline{L}_{1,2}^{\lambda}(u)}\!(v).$
Note that $x_w<(10C_\mathbb{H})^{-\lambda}x_{u^*}$
for each $w\in V(G)\setminus L$.
Thus,
\begin{eqnarray}\label{align.04}
\sum_{v\in L_1^{\lambda}\!(u), w\in N_1(v)\setminus \{u\}}\!\!\!x_w
\!\!\!&\leq&\!\!\!\!\! \sum_{v\in L_1^{\lambda}\!(u), w\in N_{L_{1,2}^{\lambda}(u)}\!(v)}\!\!\!x_{u^*}+\!\!
\sum_{v\in L_1^{\lambda}\!(u), w\in N_{\overline{L}_{1,2}^{\lambda}(u)}\!(v)}\!\!\!(10C_\mathbb{H})^{-\lambda}x_{u^*} \nonumber\\
\!\!\!&\leq&\!\!\! 2e\big(L_{1,2}^\lambda\!(u)\big)x_{u^*}
+e\big(L_1^{\lambda}\!(u),\overline{L}_{1,2}^{\lambda}(u)\big)
(10C_\mathbb{H})^{-\lambda}x_{u^*}.
\end{eqnarray}
Similarly to inequality (\ref{align.04}), we have
\begin{eqnarray}\label{align.05}
\sum_{v\in\overline{L}_1^{\lambda}\!(u), w\in N_1(v)\setminus\{u\}}\!\!\!x_w
\leq
\!\!\!\sum_{v\in\overline{L}_1^{\lambda}\!(u), w\in N_{L_{1,2}^{\lambda}(u)}\!(v)}\!\!\!x_w
+2e\big(\overline{L}_{1,2}^{\lambda}\!(u)\big)(10C_\mathbb{H})^{-\lambda}x_{u^*}.
\end{eqnarray}

By Lemma \ref{lem2.2}, we obtain
$e\big(L_{1,2}^\lambda\!(u)\big)\leq e(L^\lambda)\leq C_\mathbb{H}|L^{\lambda}|$
and
$e\big(L_1^{\lambda}\!(u),\overline{L}_{1,2}^{\lambda}\!(u)\big)
+2e\big(\overline{L}_{1,2}^{\lambda}\!(u)\big)\le 2e(G)<2C_\mathbb{H}n$.
Furthermore, by Lemma \ref{lem2.5}, 
$|L^{\lambda}|<(10C_\mathbb{H})^{\lambda-10}n$.
Summing up \eqref{align.04} and (\ref{align.05}), and combining with (\ref{align.03}),
we immediately obtain inequality (\ref{align.02}).
\end{proof}


\begin{lem}\label{lem2.7}
$|L^4|<(10C_\mathbb{H})^{6}$.
\end{lem}

\begin{proof}
We first show that
$|N_1(u)|\geq(10C_\mathbb{H})^{-5}n$ for each $u\in L^4$.
Suppose, for contradiction, that there exists $u_0\in L^4$ such that
$|N_1(u_0)|<(10C_\mathbb{H})^{-5}n$.
Setting $u=u_0$ and $\lambda=5$ in \eqref{align.02}, we obtain
\begin{eqnarray}\label{align.06}
\gamma_{\mathbb{H}}(n-\gamma_{\mathbb{H}})x_{u_0}
\leq |N_1(u_0)|x_{u_0}+\frac{4C_\mathbb{H}n}{{(10C_\mathbb{H})}^5}x_{u^*}
+\!\!\!\sum_{v\in\overline{L}_1^{5}\!(u_0), w\in N_{L_{1,2}^5\!(u_0)}\!(v)}\!\!\!x_w.
\end{eqnarray}
Notice that $\overline{L}_1^{5}(u_0)\subseteq N_1(u_0)$
and $L_{1,2}^5(u_0)\subseteq L^5$.
By Lemma \ref{lem2.2}, we have
$e\big(\overline{L}_1^{5}\!(u_0),L_{1,2}^5\!(u_0)\big)\leq
e\big(N_1(u_0)\cup L^5\big)
\leq C_\mathbb{H}\big(|N_1(u_0)|
+|L^5|\big).$
By the assumption, $|N_1(u_0)|<(10C_\mathbb{H})^{-5}n$,
and by Lemma \ref{lem2.5}, we also have $|L^{5}|\leq(10C_\mathbb{H})^{-5}n$.
Hence,
$$|N_1(u_0)|x_{u_0}
\!+\!\!\!\sum_{v\in\overline{L}_1^{5}\!(u_0), w\in N_{L_{1,2}^5\!(u_0)}\!(v)}\!\!\!x_w\leq\Big(|N_1(u_0)|
\!+\!e\big(\overline{L}_1^{5}\!(u_0),L_{1,2}^5\!(u_0)\big)\Big)x_{u^*}
\leq\frac{(1+2C_\mathbb{H})n}{{(10C_\mathbb{H})}^5}x_{u^*}.
$$
Combining inequality \eqref{align.06}, we obtain
$\gamma_{\mathbb{H}}(n-\gamma_{\mathbb{H}})x_{u_0}
\leq(1+6C_\mathbb{H}){(10C_\mathbb{H})}^{-5}n x_{u^*}.$

However, recall that $C_{\mathbb{H}}>\gamma_{\mathbb{H}}\geq1$, and
note that $x_{u_0}\ge(10C_\mathbb{H})^{-4}x_{u^*}$ as $u_0\in L^4$.
Consequently,
$\gamma_{\mathbb{H}}(n-\gamma_{\mathbb{H}})x_{u_0}\geq
\frac{7}{10}nx_{u_0}
\geq7C_\mathbb{H}{(10C_\mathbb{H})}^{-5}n x_{u^*}$,
which leads to a contradiction.
Therefore, $|N_1(u)|\geq(10C_\mathbb{H})^{-5}n$ for each $u\in L^4$.
Summing this inequality over all vertices $u\in L^{4}$, we obtain
$$|L^4|(10C_\mathbb{H})^{-5}n\le\sum_{u\in L^4}|N_1(u)|\le\!\!\sum_{u\in V(G)}|N_1(u)|
=2e(G).$$
Since $e(G)\leq C_\mathbb{H}n,$ it follows that $|L^4|\leq 2C_\mathbb{H}(10C_\mathbb{H})^{5}<(10C_\mathbb{H})^{6}$, as desired.
\end{proof}

\begin{lem}\label{lem2.8}
For each vertex $u\in L^4$,
we have $|N_1(u)|\ge(\frac{x_u}{x_{u^*}}-\frac{1}{(10C_\mathbb{H})^3})n$.
\end{lem}

\begin{proof}
Choose an arbitrary $u\in L^4$ and a minimal graph $H^*$
with respect to $\mathbb{H}$.
Then, $\gamma_\mathbb{H}=\gamma_{H^*}=|H^*|-\alpha_{H^*}-1$.
Recall that $L_i^4(u)=N_i(u)\cap L^4$,
$\overline{L}_i^4(u)=N_i(u)\setminus L^4$
and $L_{1,2}^4(u)=L_1^4(u)\cup L_2^4(u)$.
Let $L_0$ be the subset of $\overline{L}_1^4(u)$
in which each vertex has at least $\gamma_{H^*}$ neighbors in $L_{1,2}^4(u)$.

We first claim that $|L_0|\leq\varphi|H^*|$, where $\varphi=\binom{|L_{1,2}^4(u)|}{\gamma_{H^*}}$.
Indeed, if $|L_{1,2}^4(u)|\leq\gamma_{H^*}-1$,
then $L_0=\varnothing$, and we are done.
Now, consider the case where $|L_{1,2}^4(u)|\geq \gamma_{H^*}$.
Suppose, for contradiction, that $|L_0|>\varphi|H^*|$.
Since each vertex in $L_0$ has at most $\varphi$ choices for selecting 
$\gamma_{H^*}$ neighbors within $L_{1,2}^4(u)$,
the pigeonhole principle guarantees the existence of $\gamma_{H^*}$ vertices in $L_{1,2}^4(u)$
that have at least $|L_0|/\varphi>|H^*|$ common neighbors in $L_0$.
Clearly, $u\notin L_{1,2}^4(u)$
and $L_0\subseteq\overline{L}_1^4(u)\subseteq N_1(u)$.
Hence, $u$ and these $\gamma_{H^*}$ vertices must have
$|H^*|$ common neighbors in $L_0$, which implies that $G$ contains a bipartite subgraph
$G[S,T]$ isomorphic to $K_{\gamma_{H^*}+1,|H^*|}$.
As noted earlier, $|H^*|-(\gamma_{H^*}+1)=\alpha_{H^*}$.
Clearly, by contracting $\gamma_{H^*}+1$ independent edges in $G[S,T]$,
we obtain a new graph
isomorphic to the book $B_{\gamma_{H^*}+1,\alpha_{H^*}}$.
By Lemma \ref{lem2.3}, this implies that $G$ has an $H^*$ minor, leading to a contradiction.
Therefore, $|L_0|\leq\varphi|H^*|$, as claimed.

By Lemma \ref{lem2.7}, we have $|L_{1,2}^4(u)|\leq|L^4|<(10C_\mathbb{H})^{6}$,
which implies that $\varphi$ is constant.
Combining $|L_0|\leq\varphi|H^*|$, we obtain
$e(L_0,L_{1,2}^4(u))\leq|L_0||L_{1,2}^4(u)|
\leq(10C_\mathbb{H})^{-4}n$ for sufficiently large $n$.
On the other hand, by the choice of $L_0$, we know that
$e(\overline{L}_1^4(u)\setminus L_0,L_{1,2}^4(u))\leq|\overline{L}_1^4(u)\setminus L_0|(\gamma_{H^*}-1)\leq|N_1(u)|(\gamma_\mathbb{H}-1)$.
Summing up the above two inequalities gives
\begin{eqnarray}\label{align.07}
e\big(\overline{L}_1^4(u),L_{1,2}^4(u)\big)
\le|N_1(u)|(\gamma_\mathbb{H}-1)+(10C_\mathbb{H})^{-4}n.
\end{eqnarray}
Notice that
$\sum_{v\in\overline{L}_1^4(u)}\sum_{w\in N_{L_{1,2}^4(u)}(v)}x_w
\leq e\big(\overline{L}_1^4(u),L_{1,2}^4(u)\big)x_{u^*}.$
Now, by setting $\lambda=4$ in inequality \eqref{align.02} and combining with \eqref{align.07},
we obtain
\begin{eqnarray*}
 \gamma_\mathbb{H}(n-\gamma_\mathbb{H})x_u
\!\!\!&\leq&\!\!\! \Big(|N_1(u)|+\frac{2C_\mathbb{H}n}{(10C_\mathbb{H})^{6}}
 +\frac{2C_\mathbb{H}n}{(10C_\mathbb{H})^{4}}
 +|N_1(u)|(\gamma_\mathbb{H}-1)+\frac{n}{(10C_\mathbb{H})^{4}}\Big)x_{u^*}\\
\!\!\!&\leq&\!\!\! \gamma_\mathbb{H}\Big(|N_1(u)|+\frac{3C_\mathbb{H}n}{(10C_\mathbb{H})^4}\Big)x_{u^*}.
\end{eqnarray*}
This yields $|N_1(u)|\geq (n-\gamma_\mathbb{H})\frac{x_u}{x_{u^*}}-\frac{3C_\mathbb{H}n}{(10C_\mathbb{H})^4}.$
Since $\gamma_\mathbb{H}\frac{x_u}{x_{u^*}}\leq\gamma_\mathbb{H}
\leq\frac{7C_\mathbb{H}n}{(10C_\mathbb{H})^4}$ for sufficiently large $n$,
it follows that
$|N_1(u)|\ge(\frac{x_u}{x_{u^*}}-\frac{1}{(10C_\mathbb{H})^3})n$,
as desired.
\end{proof}


\begin{lem}\label{lem2.9}
For every vertex $u\in L^1$,
we have $x_u\ge\big(1-\frac{1}{2(10C_\mathbb{H})^2}\big)x_{u^*}$
and $|N_1(u)|\ge \big(1-\frac{1}{(10C_\mathbb{H})^2}\big)n$.
Moreover, we have $|L^{1}|=\gamma_\mathbb{H}$.
\end{lem}

\begin{proof}
We first show the lower bounds of $x_u$ and $|N_1(u)|$.
Suppose, for contradiction, that there exists $u_0\in L^1$ with $x_{u_0}<\big(1-\frac{1}{2(10C_\mathbb{H})^2}\big)x_{u^*}$.
By the definition of $L^1$,
we know that $x_{u_0}\geq\frac{x_{u^*}}{10C_\mathbb{H}}$.
By Lemma \ref{lem2.8}, we have
$$|N_1(u^*)|\ge\big(1-\frac{1}{(10C_\mathbb{H})^3}\big)n~~~\text{and}
~~~|N_1(u_0)|\ge\big(\frac{1}{10C_\mathbb{H}}-\frac{1}{(10C_\mathbb{H})^3}\big)n.$$
Now, let $L_i^4=N_i(u^*)\cap L^4$
and $\overline{L}_i^4=N_i(u^*)\setminus L^4$.
By Lemma \ref{lem2.7},
we have $|L^4|<(10C_\mathbb{H})^{6}<\frac{n}{(10C_\mathbb{H})^3}$
for sufficiently large $n$.
Hence, $|\overline{L}_1^4|\geq|N_1(u^*)|-|L^4|
\ge\big(1-\frac{2}{(10C_\mathbb{H})^3}\big)n,$
and thus
\begin{eqnarray}\label{align.08}
\big|\overline{L}_1^4\cap N_1(u_0)\big|\ge\big|\overline{L}_1^4\big|
+\big|N_1(u_0)\big|-n\ge\Big(\frac{1}{10C_\mathbb{H}}-\frac{3}{(10C_\mathbb{H})^3}\Big)n
\geq\frac{9n}{100C_\mathbb{H}}.
\end{eqnarray}

In view of (\ref{align.08}), $u_0$ has neighbors in $\overline{L}_1^4$.
Since $\overline{L}_1^4\subseteq N_1(u^*)$,
we have $u_0\in N_1(u^*)\cup N_2(u^*)$.
Moreover, since $u_0\in L^1\subseteq L^4$,
we conclude that $u_0\in L_{1,2}^4 $,
where $L_{1,2}^4=L_1^4\cup L_2^4$.
Now, setting $u=u^*$ and $\lambda=4$ in inequality \eqref{align.02},
we deduce that
\begin{eqnarray*}
\!&&\! \gamma_\mathbb{H}(n-\gamma_\mathbb{H})x_{u^*}\\
\!&\leq&\! \Big(|N_1(u^*)|+\frac{2C_\mathbb{H}n}{(10C_\mathbb{H})^{6}}
+\frac{2C_\mathbb{H}n}{(10C_\mathbb{H})^{4}}
+e\big(\overline{L}_1^4,L_{1,2}^4\setminus\{u_0\}\big)\Big)x_{u^*}
+e\big(\overline{L}_1^4,\{u_0\}\big)x_{u_0}\\
\!&\leq&\! \Big(|N_1(u^*)|+\frac{2.5C_\mathbb{H}n}{(10C_\mathbb{H})^4}
+e\big(\overline{L}_1^4,L_{1,2}^4\big)\Big)x_{u^*}
+e\big(\overline{L}_1^4,\{u_0\}\big)\big(x_{u_0}-x_{u^*}\big).
\end{eqnarray*}
Here, $x_{u_0}-x_{u^*}<-\frac{x_{u^*}}{2(10C_\mathbb{H})^2}$, based on the previous assumption.

From (\ref{align.07}), we know that
$e\big(\overline{L}_1^4,L_{1,2}^4\big)
\leq(\gamma_\mathbb{H}-1)|N_1(u^*)|+\frac{C_\mathbb{H}n}{(10C_\mathbb{H})^{4}}.$
Moreover, it is easy to see that $\gamma_\mathbb{H}^2\leq\frac{0.5C_\mathbb{H}n}{(10C_\mathbb{H})^{4}}$
for sufficiently large $n$.
From the inequality above, we can deduce that
$$\gamma_\mathbb{H} n
\leq\gamma_\mathbb{H}|N_1(u^*)|+\frac{4C_\mathbb{H}n}{(10C_\mathbb{H})^4}
-\frac{e\big(\overline{L}_1^4,\{u_0\}\big)}{2(10C_\mathbb{H})^2}
<\gamma_\mathbb{H}n+\frac{4C_\mathbb{H}n}{(10C_\mathbb{H})^4}
-\frac{e\big(\overline{L}_1^4,\{u_0\}\big)}{2(10C_\mathbb{H})^2},$$
which yields 
$e\big(\overline{L}_1^4,\{u_0\}\big)<\frac{8n}{100C_\mathbb{H}},$
a contradiction with (\ref{align.08}).
Hence, we conclude that $x_u\ge\big(1-\frac{1}{2(10C_\mathbb{H})^2}\big)x_{u^*}$
for each $u\in L^1$.
Furthermore,
from Lemma \ref{lem2.8}, it follows that for each $u\in L^1$,
$$|N_1(u)|\geq
\Big(1-\frac{1}{2(10C_\mathbb{H})^2}-\frac{1}{(10C_\mathbb{H})^3}\Big)n
\ge\Big(1-\frac{1}{(10C_\mathbb{H})^2}\Big)n.$$

Now, it remains to show that $|L^1|=\gamma_\mathbb{H}$.
Suppose first that
$|L^{1}|\geq\gamma_\mathbb{H}+1.$
Then $|L^{1}|\geq\gamma_{H^*}+1$,
where $H^*$ is minimal with respect to $\mathbb{H}$.
Recall that $C_{\mathbb{H}}>\gamma_{\mathbb{H}}=\gamma_{H^*}$,
and notice that every vertex $u\in L^{1}$
has at most $n/(10C_\mathbb{H})^2$ non-neighbors.
Therefore, any $\gamma_{H^*}+1$ vertices chosen from $L^{1}$ have at least $n-\frac{(\gamma_{H^*}+1)}{(10C_\mathbb{H})^2}n\geq|H^*|$ common neighbors.
Hence, $G$ contains a bipartite subgraph $G[S,T]$ isomorphic to $K_{\gamma_{H^*}+1,|H^*|}$.
Note that $|H^*|-(\gamma_{H^*}+1)=\alpha_{H^*}$.
Contracting $\gamma_{H^*}+1$ independent edges in $K_{\gamma_{H^*}+1,|H^*|}$ yields a copy of $B_{\gamma_{H^*}+1,\alpha_{H^*}}$.
By Lemma \ref{lem2.3}, this implies that $G$ contains an $H^*$ minor, a contradiction.
Hence, $|L^1|\leq\gamma_\mathbb{H}$.

Next, suppose that $|L^1|\leq\gamma_\mathbb{H}-1$.
Since $u^*\in L^{1}\setminus L_{1,2}^4$,
we have $|L^{1}\cap L_{1,2}^4|\le\gamma_\mathbb{H}-2$,
and thus $e\big(\overline{L}_1^4,L_{1,2}^4\cap L^1\big)\leq
(\gamma_\mathbb{H}-2)n$.
On the other hand, by Lemma \ref{lem2.2}, we know that
$e\big(\overline{L}_1^4,L_{1,2}^4\setminus L^{1}\big)
\leq e(G)<C_\mathbb{H}n,$
and by the definition of $L^1$, we also know that
$x_w<\frac{x_{u^*}}{10C_\mathbb{H}}$
for each $w\in L_{1,2}^4\setminus L^1.$
Now, by setting $u=u^*$ and $\lambda=4$ in inequality \eqref{align.02},
we observe that
\begin{eqnarray*}
\gamma_\mathbb{H}(n-\gamma_\mathbb{H})x_{u^*}
\!&\leq&\! \Big(|N_1(u^*)|\!+\frac{2.5C_\mathbb{H}n}{(10C_\mathbb{H})^4}+\!
e\big(\overline{L}_1^4,L_{1,2}^4\cap L^1\big)\Big)x_{u^*}
\!+\!e\Big(\overline{L}_1^4,L_{1,2}^4\setminus L^{1}\Big)
\frac{x_{u^*}}{10C_\mathbb{H}}\\
\!&\leq&\! \Big(n+\frac{n}{10}+(\gamma_\mathbb{H}-2)n
+\frac{n}{10}\Big)x_{u^*}
= \big(\gamma_\mathbb{H}-\frac45\big)nx_{u^*}.
\end{eqnarray*}
This yields $\gamma_\mathbb{H}^2\geq\frac{4}{5}n$, a contradiction.
Therefore, we have $|L^1|=\gamma_\mathbb{H}$,
completing the proof.
\end{proof}

To prove Theorem \ref{thm2.1},
it suffices to find a set $L$ of exactly $\gamma_\mathbb{H}$ vertices in $G$
such that $x_u\ge\big(1-\frac{1}{2(10C_\mathbb{H})^2}\big)x_{u^*}$
and $d_G(u)\ge\big(1-\frac{1}{(10C_\mathbb{H})^2}\big)n$
for every $u\in L$.
By Lemma \ref{lem2.9}, taking $L=L^1$ immediately yields the desired result
in Theorem \ref{thm2.1}.


\begin{thebibliography}{99}
\setlength{\itemsep}{0pt}

\bibitem{Hua1}
Y. Akama, B.B. Hua, Y.H. Su, L.L. Wang, A curvature notion for planar graphs stable under planar duality,
\emph{Adv. Math.} \textbf{385} (2021), Paper No. 107731, 44 pp.

\bibitem{Alon}
N. Alon, Eigenvalues and expanders,
\emph{Combinatorica} \textbf{6} (1986) 83--96.

\bibitem{Alon1}
N. Alon, M. Krivelevich, B. Sudakov, Complete minors and average degree: a short proof,
\emph{J. Graph Theory} \textbf{103} (2023), no. 3, 599--602.

\bibitem{Ba}
L. Babai, Automorphism groups of planar graphs. II,
\emph{Colloq. Math. Soc. J\'{a}nos Bolyai} \textbf{10} (1975) 29--84.

\bibitem{BLL}
B. Bollob\'{a}s, J. Lee, S. Letzter, Eigenvalues of subgraphs of the cube,
\emph{European J. Combin.}
\textbf{70} (2018) 125--148.

\bibitem{BN}
B. Bollob\'{a}s, V. Nikiforov, Cliques and the spectral radius,
\emph{J. Comb. Theory, Ser. B} \textbf{97} (2007) 859--865.

\bibitem{BOOT}
B.N. Boots, G.F. Royle, A conjecture on the maximum value of the principal eigenvalue of a planar graph,
\emph{Geogr. Anal.} \textbf{23} (1991) 276--282.

\bibitem{CAO}
D.S. Cao, A. Vince, The spectral radius of a planar graph,
\emph{Linear Algebra Appl.} \textbf{187} (1993)
251--257.

\bibitem{Cha}
G. Chartrand, D. Geller, S. Hedetniemi, Graphs with forbidden subgraphs,
\emph{J. Combin. Theory Ser. B} \textbf{10} (1971) 12--41.

\bibitem{Zhang}
M-Z. Chen, A-M. Liu, X-D. Zhang,
The spectral radius of minor-free graphs,
\emph{European J. Combin.} \textbf{118}
(2024) 103875.

\bibitem{CRS}
M. Chudnovsky, B. Reed, P. Seymour, The edge-density for $K_{2,t}$ minors,
\emph{J. Combin. Theory Ser. B} \textbf{101} (2011) 18--46.

\bibitem{CB1}
S. Cioab\u{a}, D.N. Desai, M. Tait,
The spectral radius of graphs with no odd wheels,
\emph{European J. Combin.} \textbf{99} (2022) 103420.

\bibitem{Con}
J.H. Conway, H. Burgiel, C. Goodman-Strauss,
The symmetries of things, A K Peters, Ltd., Wellesley, MA, 2008.

\bibitem{CP}
D. Cvetkovi\'{c}, P. Rowlinson, The largest eigenvalue of a graph: a survey,
\emph{Linear Multilinear Algebra} \textbf{28} (1990) 3--33.

\bibitem{DJS}
G.L. Ding, T. Johnson, P. Seymour, Spanning trees with many leaves,
\emph{J. Graph Theory} \textbf{37} (2001) 189--197.

\bibitem{Ding}
G.L. Ding, B. Oporowski, D.P. Sanders, D. Vertigan,
Surfaces, tree-width, clique-minors, and partitions,
\emph{J. Combin. Theory Ser. B} \textbf{79} (2000), no. 2, 221--246.

\bibitem{DM}
Z. Dvo\v{r}\'{a}k, B. Mohar, Spectral radius of finite and
infinite planar graphs and of graphs of bounded genus,
\emph{J. Combin. Theory Ser. B} \textbf{100} (2010), no. 6, 729--739.

\bibitem{EZ}
M.N. Ellingham, X.Y. Zha, The spectral radius of graphs on surfaces,
\emph{J. Combin. Theory Ser. B} \textbf{78} (2000) 45--56.

\bibitem{Gh}
L. Ghidelli, On the largest planar graphs with everywhere positive combinatorial curvature,
\emph{J. Combin. Theory Ser. B} \textbf{158} (2023), no. 2, 226--263.

\bibitem{Hadwiger}
H. Hadwiger,
\"{U}ber eine Klassifikation der Streckenkomplexe,
\emph{Vierteljschr. Naturforsch. Ges. Z\"{u}rich}
\textbf{88} (1943) 133--142.

\bibitem{FENG}
X.C. He, Y.T. Li, L.H. Feng,
Spectral extremal graphs without intersecting triangles as a minor,
\emph{Electron. J. Combin.} \textbf{31} (2024), no. 3, Paper No. 3.7, 18 pp.

\bibitem{Hi}
Y. Higuchi, Combinatorial curvature for planar graphs,
\emph{J. Graph Theory} \textbf{38} (2001), no. 4, 220--229.

\bibitem{Hong1}
Y. Hong, On the spectral radius and the genus of graphs,
\emph{J. Combin. Theory Ser. B}
\textbf{65} (1995) 262--268.

\bibitem{Hong2}
Y. Hong, Upper bounds of the spectral radius of graphs in terms of genus,
\emph{J. Combin. Theory Ser. B}
\textbf{74} (1998) 153--159.

\bibitem{Hong3}
Y. Hong, Tree-width, clique-minors, and eigenvalues, \emph{Discrete Math.}
\textbf{274} (2004) 281--287.

\bibitem{Hoory}
S. Hoory, N. Linial, A. Widgerson, Expander graphs and their applications,
\emph{Bull. Amer. Math. Soc. (N.S.)}, \textbf{43} (2006) 439--561.

\bibitem{Hua2}
B.B. Hua, Y.H. Su, The set of vertices with positive curvature
in a planar graph with nonnegative curvature,
\emph{Adv. Math.} \textbf{343} (2019) 789--820.

\bibitem{Hua3}
B.B. Hua, Y.H. Su, The first gap for total curvatures of planar graphs with nonnegative curvature,
\emph{J. Graph Theory} \textbf{93} (2020), no. 3, 395--439.

\bibitem{Huang}
H. Huang, Induced subgraphs of hypercubes and a proof of the sensitivity conjecture,
\emph{ Ann. of Math.} \textbf{(2) 190} (2019), no. 3, 949--955.

\bibitem{Jiang}
Z. Jiang, On spectral radii of unraveled balls,
\emph{J. Combin. Theory Ser. B} \textbf{136} (2019) 72--80.

\bibitem{Jiang1}
Z. Jiang, J. Tidor, Y. Yao, S. Zhang, Y. Zhao, Equiangular lines with a fixed angle,
\emph{Ann. of Math.} \textbf{(2) 194} (2021), no. 3, 729--743.

\bibitem{Ked}
K.S. Kedlaya, Outerplanar partitions of planar graphs,
\emph{J. Combin. Theory Ser. B} \textbf{67} (1996) 238--248.

\bibitem{Kos}
A.V. Kostochka, The minimum Hadwiger number for graphs with a given mean degree of vertices,
\emph{Metody Diskret. Anal.} \textbf{38} (1982) 37--58.

\bibitem{Kos1}
A.V. Kostochka, Lower bound for the Hadwiger number for graphs by their average degree,
\emph{Combinatorica} \textbf{4} (1984) 307--316.

\bibitem{Lin}
H.Q. Lin, B. Ning, A complete solution to the Cvetkovi\'{c}-Rowlinson conjecture,
\emph{J. Graph Theory} \textbf{97} (2021), no. 3, 441--450.

\bibitem{LSV}
E. Lubetzky, B. Sudakov, V. Vu, Spectra of lifted Ramanujan graphs,
\emph{Adv. Math.} \textbf{227} (2011), no. 4, 1612--1645.

\bibitem {Mader}
W. Mader, Homomorphieeigenschaften und mittlere Kantendichte von Graphen,
\emph{Math. Ann.} \textbf{174} (1967) 265--268.

\bibitem{Mader1}
W. Mader, Homomorphies\"{a}tze f\"{u}r Graphen,
\emph{Math. Ann.} \textbf{178} (1968) 154--168.

\bibitem{Ma}
P. Mani, Automorphismen von polyedrischen Graphen,
\emph{Math. Ann.} \textbf{192} (1971) 279--303.

\bibitem{MY}
J.S. Myers, The extremal function for unbalanced bipartite minors,
\emph{Discrete Math.} \textbf{271} (2003) 209--222.

\bibitem{NIKI1}
V. Nikiforov, Bounds on graph eigenvalues II,
\emph{Linear Algebra Appl.} \textbf{427} (2007) 183--189.

\bibitem{NIKI}
V. Nikiforov, The spectral radius of graphs with no $K_{2,t}$-minor,
\emph{Linear Algebra Appl.} \textbf{531} (2017) 510--515.

\bibitem{NPS}
S. Norin, L. Postle, Z. Song, Breaking the degeneracy barrier for coloring graphs with no $K_t$ minor,
\emph{Adv. Math.} \textbf{422} (2023), Paper No. 109020, 23 pp.

\bibitem{RS}
N. Robertson, P.D. Seymour, Graph minors-a survey,
Surveys in combinatorics 1985 (Glasgow, 1985), 153--171,
London Math. Soc. Lecture Note Ser., 103,
Cambridge Univ. Press, Cambridge, 1985.

\bibitem{Se}
B. Servatius, H. Servatius, Symmetry, automorphisms, and self-duality of infinite planar graphs and tilings,
in: International Scientific Conference on Mathematics. Proceedings, (\v{Z}ilina, 1998), 83--116, Univ.
\v{Z}ilina, \v{Z}ilina, 1998.

\bibitem{Tait}
M. Tait, The Colin de Verdi\`{e}re parameter, excluded minors, and the spectral radius,
\emph{J. Combin. Theory Ser. A } \textbf{166} (2019) 42--58.

\bibitem{Tait1} M. Tait, J. Tobin, Three conjectures in extremal spectral graph theory,
\emph{J. Combin. Theory Ser. B} \textbf{126} (2017) 137--161.

\bibitem{Thomason1}
A. Thomason, An extremal function for contractions of graphs,
\emph{Math. Proc. Cambridge Philos. Soc.} \textbf{95} (1984) 261--265.

\bibitem{Thomason2}
A. Thomason, The extremal function for complete minors,
\emph{J. Combin. Theory Ser. B} \textbf{81} (2001)
318--338.

\bibitem {Thomason}
A. Thomason, Disjoint complete minors and bipartite minors,
\emph{European J. Combin.} \textbf{28} (2007), no. 6, 1779--1783.

\bibitem{WAGNER}
K. Wagner, \"{U}ber eine Eigenschaft der ebenen Komplexe,
\emph{Math. Ann.} \textbf{114} (1937) 570--590.

\bibitem{WANG}
B. Wang, W.W. Chen, L.F. Fang, Extremal spectral radius of $K_{3,3} / K_{2,4}$-minor free graphs,
\emph{Linear Algebra Appl.} \textbf{628} (2021) 103--114.

\bibitem{WILF}
H. Wilf, Spectral bounds for the clique and independence numbers of graphs,
\emph{J. Comb. Theory, Ser. B} \textbf{40} (1986) 113--117.

\bibitem{WU}
B.F. Wu, E.L. Xiao, Y. Hong, The spectral radius of trees on $k$ pendant vertices, \emph{Linear
Algebra Appl.} \textbf{395} (2005) 343--349.

\bibitem{ZL}
M.Q. Zhai, H.Q. Lin, Spectral extrema of $K_{s,t}$-minor
free graphs--On a conjecture of M. Tait, \emph{J. Combin. Theory, Ser. B}
\textbf{157} (2022) 184--215.

\bibitem{ZHAI}
M.Q. Zhai, B. Wang, Proof of a conjecture on the spectral radius of $C_4$-free graphs,
\emph{Linear Algebra Appl.} \textbf{437} (2012) 1641--1647.
\end{thebibliography}
\end{document}